\documentclass[11pt,reqno]{amsart}

\usepackage[hmargin=4cm,vmargin=4cm]{geometry}

\usepackage{amssymb}
\usepackage{caption}

\usepackage{graphicx}

\usepackage{amscd}
\usepackage[leqno]{amsmath}
\usepackage{tikz}
\usepackage{tikz-3dplot}
\usepackage{tikz-cd}
\usepackage{faktor}
\usepackage{microtype}

\usepackage
[hypertexnames=false,
pdftex,pdfpagemode=UseNone,breaklinks=true,extension=pdf,colorlinks=true,linkcolor=blue,
urlcolor=blue]{hyperref}

\newtheorem{theorem}{Theorem}[section]
\newtheorem{prop}[theorem]{Proposition}
\newtheorem{lemma}[theorem]{Lemma}
\newtheorem{Corollary}[theorem]{Corollary}

\newtheorem{thm}{Theorem}

\theoremstyle{definition}
\newtheorem{defi}[theorem]{Definition}

\theoremstyle{remark}

\newtheorem{Remark}[theorem]{Remark}

\def\R{\mathbb{R}}

\newcommand{\pp}[2]{\frac{\partial#1}{\partial#2}}

\newcommand{\bigslant}[2]{{\raisebox{.2em}{$#1$}\left/\raisebox{-.2em}{$#2$}\right.}}
\newcommand{\mS}{\mathbb{S}}

\title[Periodic orbits and Birkhoff sections of SHS]{Periodic orbits and Birkhoff sections of\\ Stable Hamiltonian structures}

\author{Robert Cardona}
\address{Robert Cardona, Departament de Matem\`atiques i Inform\`atica, Universitat de Barcelona, Gran Via de Les Corts Catalanes 585, 08007 Barcelona, Spain and Centre de Recerca Matem`atica, Campus de Bellaterra, Edifici C, 08193, Barcelona, Spain}
\email{robert.cardona@ub.edu}
\thanks{R. Cardona thanks the LabEx IRMIA and the Universit\'e de Strasbourg for their support, and was partially supported by the AEI grant PID2019-103849GB-I00 / AEI / 10.13039/501100011033}

\author{Ana Rechtman}
\address{Ana Rechtman, Institut Fourier, Universit\'e Grenoble Alpes, 100 rue des math\'ematiques, 38610 Gi\`eres, France;
Institut Universitaire de France (IUF)}
\email{ana.rechtman@univ-grenoble-alpes.fr}
\urladdr{https://www-fourier.ujf-grenoble.fr/~rechtmaa/}
\thanks{Both authors are partially supported by the ANR grant CoSyDy (ANR-CE40-0014).}

\begin{document}

\begin{abstract}
Stable Hamiltonian structures generalize contact forms and define a volume-preserving vector field known as the Reeb vector field. 
We study two aspects of Reeb vector fields defined by stable Hamiltonian structures on 3-manifolds: on one hand, we classify all the examples with finitely many periodic orbits under a non-degeneracy condition; on the other, we give sufficient conditions for the existence of a supporting broken book decomposition and for the existence of a Birkhoff section. 
\end{abstract}

\maketitle

\tableofcontents 

\section{Introduction}\label{sec:intro}

The aim of this paper is to extend recent results  concerning Reeb vector fields defined by contact forms (\cite{CDR,CDHR,CM}) to Reeb vector fields defined by stable Hamiltonian structures (SHS) on closed 3-manifolds, a larger set of volume-preserving vector fields. These results concern the number of periodic orbits and the existence of either a supporting broken book decomposition or a Birkhoff section for generic sets of these vector fields. Both, broken book decompositions and Birkhoff sections, are  powerful tools for studying the dynamics of 3-dimensional flows.

Stable Hamiltonian structures naturally arise on ``stable" regular energy level sets of a Hamiltonian system, which generalize contact-type energy level sets. Stability is defined as the existence of some tubular neighborhood foliated by hypersurfaces with conjugated characteristic foliations, we refer to \cite{HZ,CV1} for more details. The SHS uniquely determines a Reeb vector field, which corresponds, up to reparametrization, to the restriction of the Hamiltonian vector field to the stable hypersurface.

\subsection{Background notions}
In this paper we are interested in the Reeb vector fields of stable Hamiltonian structures, or SHS for short. All the structures considered are $C^\infty$ unless otherwise stated.

A SHS on a closed 3-manifold $M$ is a pair of a differential 1-form $\lambda$ and a closed 2-form $\omega$, such that $\lambda\wedge \omega\neq 0$ and $\ker \omega\subset \ker d\lambda$. Observe that the second condition implies that the two kernels coincide except at the points where $d\lambda=0$. The Reeb vector field is defined as the unique vector field spanning the kernel of $\omega$ and such that $\lambda(X)=1$. In other words, up to parametrization the orbits of $X$ are defined by $\omega$ while $\lambda$ controls the parametrization.

Hofer and Zehnder \cite{HZ} introduced stability in the context of regular energy level sets of Hamiltonian systems. From a topological perspective, stable Hamiltonian structures were deeply studied in dimension 3 in seminal works by Cieliebak and Volkov \cite{CV0, CV1}, and appear as well in symplectic field theory \cite{EGH, BEHWZ, CiMo}, and other works in symplectic topology \cite{EKP, W1, LW, NW, MNW,MS}. Their dynamical properties recently attracted interest \cite{CFP1,CFP2, MP, CV2, C}, especially since the proof of the Weinstein conjecture in this context \cite{HT, Re}. In closed 3-manifolds, Reeb vector fields defined by SHS are also known as volume-preserving geodesible vector fields \cite{Re}, and arise as well as special stationary solutions to the Euler equations in hydrodynamics \cite{CV1}.

\smallskip

Consider a stable Hamiltonian structure $(\lambda,\omega)$ on a closed 3-manifold with Reeb vector field $X$. The function $f=d\lambda/\omega$ is a first integral of $X$. If $d\lambda$ is identically zero, the Reeb vector field admits a global section and the ambient manifold fibers over $\mathbb{S}^1$ by Tischler's theorem \cite{Tis}. If $f$ is either strictly positive or strictly negative, the Reeb vector field of the SHS is, up to parametrization, the Reeb vector field of the contact form $\lambda$ that we denote $X_\lambda$. In general, Cieliebak and Volkov's results \cite{CV1} imply that any SHS admits a (non-unique) \emph{structural decomposition} (see Definition~\ref{defn-strucdecom}). This is a way of decomposing the manifold $M$ into domains $U$ and $N=N_0\sqcup N_c$, whose boundary components are invariant tori of the Reeb vector field and such that $\operatorname{int}(U)\cup \operatorname{int}(N)=M$.  The closure of the connected components of these domains are manifolds with boundary with the following property. In a connected component of $U$, the Reeb vector field is integrable, i.e. the vector field is tangent to a fibration by tori. In a connected component of $N_0$, the Reeb vector field admits a global section. Finally, for a connected component of $N_c$, the Reeb vector field is the Reeb vector field of some contact form. The last region is called here the \textit{contact region}. The part $U$ where the vector field is integrable is called the \textit{integrable region} and its connected components are of the form $T^2\times I$, with $I=[0,1]$ and with the vector field tangent to each torus $T^2\times \{\cdot \}$.

\subsection{Transverse and Birkhoff sections} We use in this paper a plethora of transverse surfaces to a given vector field, introduced in this paragraph. Consider a non-singular vector field $X$ on a compact 3-manifold $M$. If $\partial M\neq \emptyset$, we assume that $X$ is tangent to $\partial M$. A {\it section} or {\it global section} of $X$ is an embedded closed surface that is everywhere transverse to $X$ and intersects all the orbits of $X$.  If the ambient manifold $M$ has boundary, a section is an embedded surface with boundary, whose boundary is mapped to $\partial M$ and satisfies the previous conditions. Observe that in these two cases, $M$ fibers over $\mathbb{S}^1$ and the dynamics of the flow of $X$ (up to parametrization) is captured by the first return map to the section. 

The previous situation cannot apply to all 3-manifolds, hence we consider more general sections of flows. 
A surface is a {\it transverse surface} if it is immersed in $M$, its interior is embedded and transverse to $X$, while its boundary is a collection of periodic orbits of $X$ or is contained in $\partial M$. In case $\partial M\neq \emptyset$, the surface is embedded near $\partial M$. A transverse surface becomes a {\it Birkhoff section} if it intersects all the orbits of $X$ in bounded time.     

Observe that we do not ask a Birkhoff section or a transverse surface to be embedded along the boundary components that are mapped to periodic orbits: we allow a connected component of the boundary to cover several times a periodic orbit and that different connected components are mapped to the same periodic orbit.

Birkhoff sections are sometimes called in the literature  {\it global surfaces of section (GSS)}, but we reserve this name, as done in other works, for Birkhoff sections whose boundary is embedded. There is a well-defined first return map in the interior of a Birkhoff section, that captures important dynamical features of the flow and allows transferring results from surface dynamics to dynamics of 3-dimensional flows and backward.

\subsection{Results on the existence of Birkhoff sections} In order to state the main theorems in the paper we need to define some sub-families of stable Hamiltonian structures. 

Consider now the periodic orbits of a non-singular vector field and the linearized Poincar\'e map of each of them. 
The vector field is non-degenerate if the eigenvalues are not equal to 1 (even when one considers the iterations of the map). Hence the linearized Poincar\'e map of a periodic orbit of a volume-preserving non-degenerate vector field is either an irrational rotation or has two real eigenvalues. The corresponding periodic orbit is called elliptic or hyperbolic, respectively. In general, non-degeneracy is a $C^\infty$-generic condition among non-singular vector fields, but not among Reeb vector fields defined by SHS. Instead, we consider {\it contact non-degenerate SHS} (Definition~\ref{def:contactND}), whose Reeb vector field is non-degenerate in the contact region of some structural decomposition. A {\it strongly contact non-degenerate}  SHS is a contact non-degenerate SHS satisfying the following additional condition:  the intersection between the stable and unstable manifolds of any hyperbolic periodic orbit in the contact region is transverse.

We consider the following sets of SHS:
\begin{enumerate}
    \item the set $\mathcal{B}$ of contact non-degenerate SHS such that, there is a structural decomposition with a contact region where the Reeb vector field is non-degenerate and in each connected component of the integrable region $U$ the slopes of the Reeb vector field are non-constant (see \textsection\ref{ss:SHS} for more details).
    \item the set $\mathcal{B}_s\subset \mathcal{B}$ of strongly contact non-degenerate SHS such that, there is a structural decomposition with a contact region where the Reeb vector field is strongly non-degenerate and in each connected component of the integrable region $U$ the slopes of the Reeb vector field are non-constant.
\end{enumerate}

\medskip

Colin, Dehornoy, Hryniewicz, and the second author in \cite{CDHR} and Contreras and Mazzuchelli in \cite{CM}, established independently that the set
of Reeb vector fields of contact forms admitting a Birkhoff section contains an open and dense set in the $C^\infty$-topology. These results are based on the construction of broken book decompositions for non-degenerate Reeb vector fields of a contact form. Broken book decompositions (see Definitions ~\ref{defn:BBD} and~\ref{defn:BBD2}) were introduced in \cite{CDR} and provide a finite collection of transverse surfaces that intersect all orbits of the vector field. In analogy, our main result is

\begin{thm}\label{thm:mainBirk}
On any closed connected 3-manifold, 
\begin{itemize}
\item[-] every Reeb vector field of a SHS in  a $C^2$-neighborhood of the set $\mathcal{B}$ is supported by a broken book decomposition;
\item[-] every Reeb vector field of a SHS in a $C^2$-neighborhood of the set $\mathcal{B}_s$ admits a Birkhoff section;
\item[-] the set $\mathcal{B}_s$ is $C^1$-dense among the set of SHS.
\end{itemize}
\end{thm}

The first two parts of this theorem are proved in \textsection\ref{s:SHSsections}, they correspond respectively to Theorems~\ref{thm:bbook} and \ref{thm:Bsection}.
The last statement follows from results in \cite{CV1} and is explained in \textsection\ref{ss:density}. For the proof of the first two parts, we construct a broken book decomposition or a Birkhoff section for the Reeb vector field of a SHS in $\mathcal{B}$ and $\mathcal{B}_s$, respectively. We then extend the construction to a neighborhood of these sets using results from \cite{CV2}.

We now consider the set $\mathcal{CSHS}(M)$ of vector fields that are a (positive or negative) reparametrizations of the Reeb vector field defined by some SHS on a closed 3-manifold $M$ and $\mathfrak{X}(M)$ the set of vector fields on $M$. If we allow to perturb in this set, instead of perturbing in the set of pairs of differential forms that define a SHS, we obtain a stronger statement.

\begin{thm}\label{thm:generic}
 Let $M$ be a closed connected 3-manifold. There exists a $C^\infty$-dense subset $A\subset \mathcal{CSHS}(M)$ such that every vector field in a $C^1$-open neighborhood of $A$ inside $\mathfrak{X}(M)$ admits a Birkhoff section.
\end{thm}
Observe that, in particular, it follows that a $C^\infty$-generic flow in $\mathcal{CSHS}(M)$ admits a Birkhoff section. The same statement holds if all the vector fields to preserve a given volume form, see Remark \ref{rem:fixvol}. Another genericity statement that we show is that the set of SHS that admit a Birkhoff section contains a $C^1$-open and $C^1$-dense set among the SHS, see Theorem~\ref{thm:C1generic}. This differs with the previous statement because we perturb the SHS instead of the vector field.

{
 \begin{Remark}
     Even if Theorem~\ref{thm:generic} might apparently imply that there should be a $C^\infty$-dense set of SHS whose Reeb vector field admits a Birkhoff section, this is not the case. Notice that given a vector field in $X\in \mathcal{CHSH}(M)$ and a SHS $(\lambda,\omega)$ such that $\lambda(X)>0, \iota_X\omega=0$, Theorem~\ref{thm:generic} implies that arbitrarily close to $X$ there is some $Y$ with the following property: there is a SHS$(\lambda',\omega')$ (and even $\omega'$ can be chosen arbitrarily $C^\infty$-close to $\omega$) such that $\lambda'(Y)>0$. However, the 1-form $\lambda'$ will in general not be $C^\infty$-close to $\lambda$, but only $C^1$. In other words, the Reeb vector field of $(\lambda,\omega)$ and $(\lambda',\omega')$ are not $C^\infty$-close as vector fields, but the 1-dimensional foliations they define are $C^\infty$-close. Hence, a suitable choice of vector fields defining these foliations are $C^\infty$-close vector fields.
 \end{Remark}
}

The main new ingredient behind the existence statements of Theorem~\ref{thm:mainBirk} is the concept of a {\it helix box}. The analogous theorems for Reeb vector fields of contact structures give a Birkhoff section and a broken book decomposition in the contact parts of a suitable SHS, after analyzing the behavior near the boundary of these parts. On the other hand, in the suspension part there is, by definition, a global section. The main difficulty lies in pasting these surfaces along the integrable parts that are of the form $T^2\times I$: we build local Birkhoff sections with boundary that allow pasting any  homology class in $T^2\times\{0\}$  with any homology class in $T^2\times\{1\}$. These local Birkhoff sections are what we call {\it helix boxes} and are constructed in \textsection\ref{sec:gluing}. 

\medskip

In \textsection\ref{sec:analytic}, we consider the set of SHS for which the function $f=d\lambda/\omega$ is real analytic and non-constant. Thus, the Reeb vector field admits a non-trivial real analytic first integral and  we can prove, without any non-degeneracy assumption, that they  admit a \emph{Birkhoff splitting}. 
Roughly speaking a Birkhoff splitting is defined as a connected Birkhoff section of the flow on the complement of a finite collection of invariant tori (see Definition~\ref{def:Bsplit}). Birkhoff splittings are constructed using a combination of our techniques and the analysis done in \cite{CV2}.

\subsection{Results on periodic orbits of Reeb vector fields}
We now introduce our results on the number of periodic orbits. Reeb vector fields of SHS might not have periodic orbits, but it is known that there is always one if the ambient manifold is not a torus bundle over the circle \cite{HT,Re}. We describe all the possible examples without periodic orbits:

\begin{thm}\label{thm:main1}
Let $X$ be an aperiodic Reeb vector field of a stable Hamiltonian structure $(\lambda,\omega)$ on a closed connected 3-manifold $M$. Then one of the following holds:
\begin{enumerate}
\item $M=T^3$ or a positive parabolic torus bundle over $\mS^1$, the Reeb vector field admits a torus global section and it is orbit equivalent to the suspension of an aperiodic symplectomorphism of the torus,
\item $M$ is a positive hyperbolic torus bundle over $\mS^1$, the Reeb vector field does not admit a global section, and after cutting along an invariant torus, it is orbit equivalent to the suspension of a 
pseudorotation of the annulus which is a rigid rotation near the boundary and has a quadratic irrational rotation number.
\end{enumerate}
If $(\lambda,\omega)$ is assumed to be analytic, then only the first case occurs.
\end{thm}

A pseudorotation of the annulus is an area-preserving map without periodic points, the definition is discussed in \textsection\ref{ssec-surfacedyn}. 

Under the contact non-degeneracy assumption, we can characterize as well those SHS whose Reeb vector field has only finitely many periodic orbits. The following theorem is an analog of the two or infinitely many periodic orbits theorem for non-degenerate contact Reeb vector fields proved in \cite{CDR} (see also \cite{HWZ, CgHP}).

{
\begin{thm}\label{thm:main2}
Let $M$ be a closed connected $3$-manifold, and $(\lambda,\omega)$ a contact non-degenerate SHS whose Reeb vector field has at least one periodic orbit. Then exactly one of the following holds:
\begin{itemize}
    \item[-] The Reeb vector field has infinitely many periodic orbits.
    \item[-] The Reeb vector field is orbit equivalent to the suspension of a symplectomorphism of a surface $\Sigma_g$ with finitely many periodic points.
    \item[-] The ambient manifold $M$ is the 3-sphere or a lens space, there are exactly two closed Reeb orbits and they are the core circles of a genus one Heegaard splitting of $M$.
\end{itemize}
Furthermore, the same conclusion holds for any  SHS that has a periodic Reeb orbit and belongs to a $C^2$-neighborhood of a contact non-degenerate SHS.
\end{thm}
}

We point out that the proof of Theorem~\ref{thm:main2} involves proving that an area preserving diffeomorphism of an oriented surface with boundary, without periodic points along the boundary, has infinitely many periodic points (confer Theorem~\ref{thm:surfirra}) except in the disk and the annulus. This proof is inspired by the ideas in \cite{Lec1}. It is known that symplectomorphisms of closed oriented surfaces with finitely many periodic points (and at least one) only exist in finite order isotopy classes of diffeomorphisms of a surface \cite[Theorems 1.2 and 1.3]{Lec1} (see also \cite{AZT}). Those isotopy classes characterize the surface bundles where contact non-degenerate SHS with finitely many periodic orbits do exist.

\subsection{Structure of the paper} We start presenting known facts about stable Hamiltonian structures and surface dynamics in \textsection\ref{sec:prelim}, in particular, we present the notion of structural decomposition of a SHS and Definition~\ref{def:contactND} of contact non-degenerate SHS. We then study aperiodic SHS in \textsection\ref{s:Aper} and the contact non-degenerate examples with finitely many periodic orbits in \textsection\ref{s:finite}, proving Theorems~\ref{thm:main1} and \ref{thm:main2} respectively. In \textsection\ref{s:SHSsections} we give the proof for the first two parts of Theorem~\ref{thm:mainBirk}, starting with the construction of helix boxes in \textsection\ref{sec:gluing} that establishes a very general result for Birkhoff sections in integrable regions. In \textsection\ref{sec:gen}, we explain why the last statement in Theorem~\ref{thm:mainBirk} holds, discuss different genericity results for the existence of Birkhoff sections and establish Theorem~\ref{thm:generic}. Finally, in \textsection\ref{sec:analytic}, we establish the existence of Birkhoff splittings for SHS whose Reeb vector field admits an analytic integral. 

\medskip

\textbf{Acknowledgements}: We are grateful to Patrice Le Calvez, for discussions concerning surface dynamics and many ideas to prove Theorem~\ref{thm:surfirra}, and to Francisco Torres de Lizaur for useful discussions. We also thank Pierre Dehornoy, for pointing out a detail of helix boxes that allowed us to simplify the statements of Theorems~\ref{thm:Bsection} and \ref{thm:bbook}. We thank the referees for all the corrections and comments that allowed us to improve the quality of this work. 

\section{Preliminaries}\label{sec:prelim}

In this section, we recall some necessary definitions and results that will be used throughout this work.

\subsection{Stable Hamiltonian structures}\label{ss:SHS}

As mentioned in the introduction, stable Hamiltonian structures generalize contact forms and are defined on closed odd-dimensional manifolds. In this work, we only consider the 3-dimensional case.

\begin{defi}
A stable Hamiltonian structure (SHS) on an oriented 3-manifold $M$ is a pair $(\lambda,\omega)$ where $\lambda \in \Omega^1(M)$ and $\omega \in \Omega^2(M)$ such that
$$\lambda\wedge \omega>0, \qquad d\omega=0 \qquad \mbox{and} \qquad \ker \omega \subset \ker d\lambda.$$
\end{defi}
It uniquely determines a Reeb vector field $X$ by the equations
$$ \lambda(X)=1,\qquad  \iota_X\omega=0.$$
 Observe that $X$ preserves the volume form $\lambda\wedge \omega$ and the plane field $\xi=\ker \lambda$, which is transverse to $X$. Only in dimension three, Reeb vector fields of SHS can be equivalently characterized as geodesible volume-preserving vector fields normalized to be of unit length \cite{Re}. The function $f=d\lambda/\omega$ is a first integral of the flow of $X$.

If the 2-form $d\lambda$ is never zero, $\lambda$ is a contact form and $X$ is modulo reparametrization the Reeb vector field $X_\lambda$, while if $d\lambda\equiv 0$ then $M$ fibers over the circle and each fiber is a section of $X$. This implies that the flow of $X$ is  orbit equivalent to
the suspension of an area-preserving diffeomorphism of a closed surface.

Given a  SHS $(\lambda,\omega)$ a 1-form $\widetilde{\lambda}$ is another stabilizing form of $\omega$ if $(\widetilde{\lambda}, \omega)$ is again a SHS. In this situation the Reeb vector field of $(\widetilde{\lambda},\omega)$ is a positive multiple of the Reeb vector field of $(\lambda,\omega)$.

{For a stable Hamiltonian structure $(\lambda,\omega)$ with proportionality function $f=\frac{d\lambda}{\omega}$, an \textit{integrable region} $D\subset M$ is a domain diffeomorphic to $T^2\times I$, where $I=[0,1]$, such that the Reeb vector field of $(\lambda,\omega)$ integrates to a subfoliation of the trivial foliation by tori in $D$. We say that $(\lambda,\omega)$ is $T^2$-invariant in coordinates $(x,y,t)$ of $T^2\times I$ when $\omega$ is of the form
\begin{equation}\label{eq-omegaonreg}
    \omega= h_1(t) dt\wedge dx + h_2(t) dt\wedge dy,
\end{equation}
and $\lambda$ of the form
\begin{equation}\label{eq-lambdaonreg}
\lambda = g_1(t) dx + g_2(t) dy + g_3(t)dt.
\end{equation}
We point out that whenever we work locally in some integrable region $T^2\times I$, up to slightly shrinking $I$ we might assume that $g_3(t)\equiv 0$ by \cite[Lemma 3.9]{CV1}. Whenever the size of $I$ does not need to be fixed, we will make this assumption without saying it, to lighten a bit the notation. We introduce now an important notion for stable Hamiltonian structures.
\begin{defi}\label{defn-strucdecom}
    A \emph{structural decomposition} of a SHS $(\lambda,\omega)$ is a decomposition of $M$ into a compact $3$-dimensional submanifold $N=N_c \sqcup N_0$ (possibly with boundary, possibly disconnected) of $M$, invariant under the Reeb flow; and a (possibly empty) disjoint union $U=U_1\sqcup ...\sqcup U_r$ of integrable regions such that
\begin{enumerate}
    \item $\operatorname{int}U \cup \operatorname{int}N=M$;
    \item On each component $N_{c,j}$ of $N_c$, the function $f$ is non-vanishing;
    \item On each component $N_{0,k}$ of $N_0$ there is a 1-form $\beta_k \in \Omega^1(N_{0,k})$ such that $\beta_k\wedge \omega|_{N_{0,k}}\neq 0$ and $d\beta_k=0$;
    \item for each $U_i\cong T^2\times I$, the set $U_i\setminus N$ is of the form $T^2\times (a,b)$ with $a>0$ and $b<1$;
    \item on each $U_i\cong T^2\times I$ there are coordinates $((x,y),t)$ for which the stable Hamiltonian structure $(\lambda,\omega)$ is $T^2$-invariant, moreover in $U_i\setminus N$ the torus fibers are regular level sets of $f=f(t)$.
\end{enumerate}
\end{defi}
}
Notice that a given SHS might admit several structural decompositions that can be very different.
We make a few observations on Definition~\ref{defn-strucdecom}. First, by condition (1), each boundary component of $N$ is contained in $U$ and hence belongs to an integrable region. This implies that inside $N$ each boundary component has a neighborhood that is foliated by 2-dimensional tori tangent to the Reeb vector field and in these neighborhoods $X$ is fiberwise linear. Condition (3) implies that $N_0$ fibers over $\mathbb{S}^1$ and that the flow of the Reeb vector field is orbit-equivalent to a suspension of a symplectomorphism of a surface (with boundary if $N_0\neq M$). The domain $N_c$ splits into two disjoint domains $N_+, N_-$ where $f=\frac{d\lambda}{\omega}$ is respectively positive and negative. In $N_+$ and $N_-$, the 1-form $\lambda$ is respectively a positive or a negative contact form and the Reeb vector field $X$ of the SHS is colinear with the Reeb vector field $X_\lambda$ of the contact form $\lambda$. We call $N_c$ the {\it contact region} of such a decomposition.

The $T^2$-invariance of $(\lambda, \omega)$ implies that in each connected component of $U$ with the coordinate system  in part (5), 
the Reeb vector field $X$ is a linear vector field  and we can write
\begin{equation}\label{eq-Xonregular}
    X=a_1(t)\frac{\partial}{\partial x}+a_2(t)\frac{\partial}{\partial y}.
 \end{equation}
 
{The theory developed in \cite{CV1}, concretely \cite[Theorem 3.3]{CV1} and the key \cite[Proposition 3.23]{CV1}, implies that any SHS admits a structural decomposition. 
\begin{theorem}\cite{CV1}
    Any SHS admits a structural decomposition.
\end{theorem}
We state here a version of \cite[Theorem 4.1]{CV1}, which proves that any stable Hamiltonian structure admits, up to changing the stabilizing 1-form by a $C^1$-perturbation, a very special structural decomposition.
\begin{theorem}[SHS structure theorem]\label{thm:struc}
Let $(\lambda,\omega)$ be a stable Hamiltonian structure on a closed 3-manifold $M$. Then there exists a stabilizing 1-form $\widetilde \lambda$ that is $C^1$-close to $\lambda$ such that $(\widetilde \lambda, \omega)$ admits a structural decomposition with the property that $\widetilde f$ is constant and non-vanishing in $N_c$ and equal to zero in $N_0$.
\end{theorem}
The fact that $\widetilde f$ is constant in $N$ is crucial to allow for perturbations of the SHS in $N$, and to prove that certain non-degenerate SHS are dense. The statement above is slightly different from \cite[Theorem 4.1]{CV1}, but follows from its proof. Indeed, the function $\widetilde f=\frac{d\widetilde \lambda}{\omega}$ constructed in the proof of \cite[Theorem 4.1]{CV1} has exactly $N$ as preimage of its singular values, and thus $\tilde f$ is regular (and its level sets are torus fibers) in $U_i\setminus N$. Observe that the 2-form $\omega$ has not changed, and thus the Reeb vector field of $(\widetilde \lambda,\omega)$ is just a reparametrization of the Reeb vector field of $(\lambda,\omega)$ by a function $C^1$-close to one.}

If one asks that the Reeb vector field of a SHS is  non-degenerate on $M$ (e.g. as in \cite{MS}) and $U$ is non-empty, the vector field has to be of constant irrational slope on each $U_i$. This is a very strong condition that might not be dense in some reasonable topology in the set of SHS. To apply results from \cite{CDR}, we need that in the contact region $N_c$ the Reeb vector field is non-degenerate, but we do not need the vector field to be non-degenerate everywhere. Hence we define

\begin{defi}\label{def:contactND}
A stable Hamiltonian structure $(\lambda,\omega)$ is contact non-degenerate if there exists a structural decomposition of $(\lambda,\omega)$ such that $\lambda$ is a non-degenerate contact form in $N_c$.
\end{defi}

Recall that a contact form is non-degenerate if its Reeb vector field is non-degenerate, and this property is satisfied by a dense set of contact forms in the  $C^\infty$-topology. In \cite[p 375]{CV1}, Cieliebak and Volkov introduce two different conditions of non-degeneracy for a Reeb vector field of a SHS. One is that of Morse non-degeneracy, meaning that the Reeb flow is everywhere non-degenerate. This definition is stronger than Definition~\ref{def:contactND}, since it requires non-degeneracy in all $M$, which is in general larger than just the contact region $N_c$ for a choice of $N$. The other non-degeneracy condition is that of Morse-Bott non-degeneracy in all $M$, an analog of the Morse-Bott non-degeneracy of contact forms. Being Morse-Bott is less strong than the contact non-degeneracy along a contact region (for some structural decomposition), but stronger than contact non-degeneracy in the complement of that contact region. Indeed, contact non-degeneracy does not assume anything away from a contact region of some structural decomposition. Both Morse-Bott non-degeneracy and contact non-degeneracy are dense conditions in the set of SHS in the $C^1$-topology.

If $(\lambda,\omega)$ is contact non-degenerate, the periodic orbits of its Reeb vector field $X$ in $N_c$ are either elliptic or hyperbolic, and near each connected component of $\partial N$ there is a foliation by invariant 2-tori in which $X$ is colinear to a linear vector field of irrational slope. In particular, the periodic orbits of $X$ are far away from $\partial N$. Observe also that if $\gamma$ is a hyperbolic periodic orbit in $N_c$, it has a stable and an unstable manifold denoted respectively by $W^u(\gamma)$ and $W^s(\gamma)$ which are contained in $N_c$.
{
\begin{defi}
A contact non-degenerate stable Hamiltonian structure $(\lambda,\omega)$ is contact strongly non-degenerate if the intersections between stable and unstable manifolds of the hyperbolic periodic orbits in $N_c$ are transverse.    
\end{defi}
}
 
  The arguments in \cite{CV1} can be adapted to show that contact strongly non-degenerate SHS are $C^1$-dense as well, see \textsection\ref{ss:density}.

\subsection{Torus bundles over $\mathbb{S}^1$}
The classification of torus bundles over $\mathbb{S}^1$ will play an important role in the proof of Theorem~\ref{thm:main1}. Let us recall the main properties that we will need. An (orientable) torus bundle over $\mathbb{S}^1$ is obtained by considering the mapping torus of an orientation-preserving diffeomorphism $\varphi:T^2\to T^2$ of a 2-torus. The isotopy class of $\varphi$ in the group of orientation preserving diffeomorphisms $\operatorname{Diff}(T^2)$ is determined by the action on the first homology group of $T^2$, which is given by an element $A\in SL(2,\mathbb{Z})$. 
The conjugacy class of $A$ in $SL(2,\mathbb{Z})$ defines the torus bundle up to homeomorphism. These classes are divided into three groups depending on the trace of $A$:
\begin{itemize}
\item[-] If $|\operatorname{tr}(A)|<2$, there are two conjugacy classes for each possible value of the trace: $-1, 0$ and $1$. The torus bundles obtained via these matrices are called \textit{elliptic} torus bundles. In this case, the matrix is not diagonalizable.
\item[-] If $|\operatorname{tr}(A)|=2$, there are two $\mathbb{Z}$-families of conjugacy classes given by matrices of the form
$$\begin{pmatrix}
1 & n \\
0 & 1
\end{pmatrix}, \quad \begin{pmatrix}
-1 & n \\
0 & -1
\end{pmatrix}, \text{ with } n\in \mathbb{Z}.$$
We call the torus bundles obtained using these matrices \textit{positive parabolic} and \textit{negative parabolic} torus bundles respectively. In this case, both eigenvectors have rational slope.
\item[-] If $|\operatorname{tr}(A)|>2$,  the matrix always has eigenvalues $\lambda$ and $\frac{1}{\lambda}$ for some $\lambda\neq 0$. The corresponding eigenvectors have irrational slope. We call the torus bundles obtained using a matrix in this class \textit{hyperbolic} torus bundles. We call it a positive or negative hyperbolic bundle if the trace of $A$ is positive or negative respectively.
\end{itemize}

\subsection{Surface dynamics}\label{ssec-surfacedyn}
Let us recall a few definitions that we will need concerning homeomorphisms of compact surfaces. For details, we refer to \cite{BCLP, Lec1}.

Let $f: \mathbb{A}\to \mathbb{A}$ be a homeomorphism of the closed annulus $\mathbb{A}=\mS^1\times I$ isotopic to the identity. Let $\widetilde{\mathbb{A}}=\mathbb{R}\times I$ be 
the universal cover of $\mathbb{A}$ and $\widetilde f: \widetilde{\mathbb{A}} \to \widetilde{\mathbb{A}}$ be a lift of $f$. Denote by $\pi_1: \widetilde{\mathbb{A}} \to \mathbb{R}$ the projection into the first factor. Given an $f$-invariant compactly supported Borel probability measure $\mu$, we define the rotation number of $\mu$ as
$$ \operatorname{rot}_{\widetilde f}(\mu)=\int_D (\pi_1\circ \widetilde f - \pi_1)d\mu, $$
where $D$ is any fundamental domain of the covering. The rotation set $\operatorname{rot}(\widetilde f)\subset \mathbb{R}$ is the set of rotation numbers for all invariant measures, and defines a compact interval. For two different lifts, the rotation set only differs by an integer number. 

A special class of maps of the annulus is the one formed by pseudorotations, which can be characterized by their rotation set:
a pseudorotation of the annulus is an area-preserving homeomorphism of $\mathbb{A}$ isotopic to the identity and whose rotation set reduces to a single irrational number.

One can define as well pseudorotations of the closed disk and the sphere. We give a definition, equivalent to the fact that the rotation set is a single irrational number, in terms of the fixed and periodic points:
\begin{defi}
    A pseudorotation is a homeomorphism with no wandering points of a surface that is isotopic to the identity map and such that:
    \begin{itemize}
        \item[-] if the surface is the annulus, it has no periodic points;
        \item[-] if the surface is the disk, it has one fixed point and no other periodic points;
        \item[-] if the surface is the sphere, it has two fixed points and no other periodic points.
    \end{itemize}
\end{defi}

Recall that if $f:X\to X$ is a homeomorphism of a topological space, a point $x\in X$ is {\it wandering} if there is an open set $U$ containing $x$ and a $N>0$ such that for all $n>N$ we have $f^n(U)\cap U=\emptyset$. Otherwise, $x$ is a {\it non-wandering} point. When $f$ preserves an area form and $X$ is compact, every point is non-wandering.  The following result combines two classical results of Franks \cite{F1, F2} on the existence of periodic orbits of surface homeomorphisms with no wandering points. We point out that Franks theorems are stated for homeomorphisms that are chain transitive, which is a stronger condition than having all points non-wandering (see \cite[Proposition 1.2]{F1}).

\begin{theorem}\label{thm:Fk}
    Let $F:\Sigma \rightarrow \Sigma$ be an orientation preserving homeomorphism with no wandering points, where $\Sigma$ is a sphere, a disk, or a closed annulus. Then either $F$ is a pseudorotation or it admits infinitely many periodic points. If the rotation set contains more than a point, there are periodic points with arbitrarily large periods in a compact subset $K$ in the interior of $\Sigma$.
\end{theorem}

{A pseudorotation of the disk has a fixed point, that can be blown up to obtain a pseudorotation of an annulus. Likewise, in the case of the sphere, there are two fixed points that can be blown up to obtain also a pseudorotation of an annulus. Thus the statement above reduces to the case of the annulus.}{ Let us justify why the periodic points of an arbitrarily large period lie in a compact set $K$ in the interior of the closed annulus $\Sigma$. The fact that the rotation set is not a single point allows us to choose an interval $[c,d]$ in the rotation set that has a non-empty interior and does not contain the rotation number of any of the boundary components of $\Sigma$. The periodic points of arbitrarily large periods can be chosen to have rotation numbers in $[c,d]$ by \cite[Corollary 2.4]{F1}. Then, by continuity of the rotation number, these orbits will all lie at a uniform positive distance from the boundary.} 

\medskip

Another class of homeomorphisms that we will use are the so-called Dehn twist maps, defined on closed surfaces of genus at least $2$.
\begin{defi}
Let $\Sigma$ be a closed surface of genus $g\geq 2$. A Dehn twist map is an orientation preserving homeomorphism $h:\Sigma \to \Sigma$ such that:
\begin{itemize}
\item[-] there is a finite family of pairwise disjoint invariant essential closed annuli $(A_i)_{i=1,...,k}$,
\item[-] no connected component of $\Sigma\setminus \bigcup_{i=1}^k A_i$ is an annulus,
\item[-] $h$ fixes every point in $\Sigma \setminus \bigcup_{i=1}^k A_i$,
\item[-] the map $h|_{A_i}$ is conjugated to $\tau^{n_i}$, where $n_i\neq 0$ and $\tau$ is a homeomorphism of $\mS^1\times [0,1]$ that lifts to $\widetilde \tau (x,y)=(x+y,y)$.
\end{itemize}
\end{defi}

\section{Characterization of aperiodic SHS}\label{s:Aper}

Given a stable Hamiltonian structure $(\lambda,\omega)$, it is immediate that if $f=\frac{d\lambda}{\omega}$ never vanishes, then $\lambda$ is a contact form and so $X$ cannot be aperiodic by Taubes' theorem \cite{Taub}. Hence, aperiodic examples only occur when $f$ vanishes somewhere and we start by analyzing the cases in which $f$ is not identically zero. In this section, we provide a dynamical and topological characterization of counterexamples to the Weinstein conjecture for SHS, proving Theorem~\ref{thm:main1}. As an intermediate step, we study contact Reeb dynamics in the 3-manifold with boundary $N_c$ associated to a structural decomposition as in Definition~\ref{defn-strucdecom} of a SHS $(\lambda,\omega)$. 


\subsection{{Reeb vector fields of a SHS in $N_c$} 
}
Denote by $I$ the interval $[0,1]$ and by $\mathbb{A}=\mS^1\times I$ the closed annulus. 
In order to study Reeb vector fields of SHS for which $f$ is not identically zero, we start by

\begin{lemma}\label{lem:reebpseudo}
Let $M$ be a closed oriented 3-manifold endowed with $(\lambda,\omega)$  a contact non-degenerate SHS with respect to a structural decomposition such that $N_c\neq M$. If the Reeb vector field $X$ of $(\lambda, \omega)$  has finitely many periodic orbits in $N_c$,  then either $N_c\cong T^2\times I$ and $X$ is orbit equivalent to the suspension of a pseudorotation of the annulus, or $N_c\cong \mS^1\times D^2$ and $X$ is orbit equivalent to the suspension of a pseudorotation of the disk.
\end{lemma}

\begin{proof}
 Since near the boundary of $N_c$ the SHS $(\lambda, \omega)$ is $T^2$-invariant (see \textsection\ref{ss:SHS}), we have that each connected component of $\partial N_c$ admits a neighborhood $V\cong T^2\times [0,\delta]$ with coordinates $(x,y,t)$ so that
\begin{equation}\label{eq:reebinv}
X=a_1(t) \pp{}{x} + a_2(t) \pp{}{y}. 
\end{equation}
{Observe that since we assume that $X$ has finitely many periodic orbits in $N_c$, we have that $a_2(t)/a_1(t)$  is a constant irrational number. The arguments in \cite[Section 5.2]{HT} imply that the contact form $\lambda$} can be assumed to be
\begin{equation}\label{eq:contform}
 \lambda= \frac{1}{2} \rho(t)^2 (a_2dx-a_1dy) + \alpha_2 dy, 
\end{equation}
 where $\rho(t)^2$ is a smooth function on $[0,\delta]$ and $\rho(t)$ is smooth and strictly increasing on $(0,\delta]$. We can collapse the boundary torus to a circle by identifying those points that have the same $x$ coordinate along the torus $t=0$. Then $(\rho,x)$ become polar coordinates on the disk and $y$ is the coordinate of the $S^1$ direction. 
 
 Doing this along each component of the boundary we obtain a closed 3-manifold $N_1$, and $\lambda$ becomes a smooth contact form $\lambda_1$ on $N_1$ just by using the expression \eqref{eq:contform} along each solid torus $V_1$ obtained from $V$. Let $\sigma_i$ be the circles that we constructed by collapsing the boundary tori. The Reeb vector field $X_{\lambda_1}$ of $\lambda_1$ in $N_1\setminus \{\sigma_i\}$  coincides with $X$ in the interior of $N_c$. 

The Reeb vector field $X_{\lambda_1}$ is non-degenerate and has finitely many periodic orbits, by hypothesis. Then \cite[Theorem 1.2]{CDR} implies that the manifold $N_1$ is a sphere or a lens space and $X_{\lambda_1}$ has exactly two periodic orbits that we denote by $\gamma_1$ and $\gamma_2$.  By \cite[Theorem 1.2]{HT}, the periodic orbits are the core circles of the solid tori of a genus one Heegaard splitting of $N_1$. This implies that either $N_c\cong T^2\times I$ or $N_c \cong \mS^1 \times D^2$. Finally, by \cite[Corollary 1.10]{CGHHL}, each closed Reeb orbit of $X_{\lambda_1}$ bounds a disk-like Birkhoff section (one could alternatively find an annulus-like Birkhoff section induced by the projection of holomorphic cylinders, see \cite[Section 4.6]{HT}). 

At least one of these two periodic orbits coincides with a collapsed boundary component of $N_c$. Assume without loss of generality that $\gamma_2$ corresponds to a collapsed boundary component. Consider the disk-like section bounded by $\gamma_2$ and denote the first-return map by $h$. The map $h$ is smooth since there is an invariant foliation by circles near the boundary. Moreover, $h$ is conjugated to a rigid rotation in each such circle. Since $h$ has finitely many periodic points, Theorem~\ref{thm:Fk} implies that $h$ is conjugated to a pseudorotation of the disk.

If $N_c\cong \mS^1\times D^2$, then $X_{\lambda_1}$ is tangent to a foliation by invariant tori near the boundary that corresponds to the circle $\gamma_2$. We can restrict $h$ to a disk $D'$ whose boundary is an invariant circle close enough to $\gamma_2$, and hence $X_{\lambda_1}$ is orbit equivalent to the suspension of $h|_{D'}$, a pseudorotation of the disk, and we conclude in this case. 

In the other case, the original manifold is $N_c\cong T^2\times I$, which means that near the boundary $\gamma_2$ and near the fixed point of $h$ there is a foliation by invariant circles of $h$. Along those circles, the diffeomorphism is conjugated to a rigid rotation. By restricting $h$ to some annulus $\mathbb{A}$ bounded by two invariant circles, one close enough to the fixed point and one close enough to $\gamma_2$, we obtain an area-preserving diffeomorphism $h|_\mathbb{A}$
of the closed annulus.  It follows that $X_{\lambda_1}$ is orbit equivalent to the suspension of $h$, and the same holds for $X$ in $N_c$. We conclude that $X$ is orbit equivalent to the suspension of a pseudorotation of the annulus.
\end{proof}

\subsection{A global section up to cutting along a torus}

In this section, we prove that aperiodic SHS always admit a global section if one allows to cut open the ambient manifold along an invariant torus. A key fact is that the suspension of a pseudorotation of an annulus $\mathbb{A}$ that can be extended in a particular way to a homeomorphism of the torus admits global sections inducing any homology class on the boundary. 

 Following \cite{K1} we recall some definitions needed to apply Fried's theory of global sections \cite{Fr}. Let $g$ be a diffeomorphism of the torus $T^2$, that we assume to be in the path-connected component of the identity. Given a lift $\widetilde g:\mathbb{R}^2 \rightarrow \mathbb{R}^2$ of $g$, the rotation vector of $\widetilde g$ at $\widetilde p \in \mathbb{R}^2$ is defined as
 $$\rho(\widetilde p, \widetilde g)=\lim_{n\rightarrow \infty} \frac{\widetilde g^n(\widetilde p)-\widetilde p}{n},$$
 whenever the limit is defined. The rotation vector at a point for two different lifts of $g$ differs by a fixed vector in $\mathbb{Z}^2$. For a fixed lift, the rotation set $\rho(\widetilde g)$ is the set of rotation vectors obtained at any point in $\mathbb{R}^2$. We call a diffeomorphism $g$ a pseudorotation of the torus if the rotation set of $\widetilde g$ is a unique vector whose coordinates are independent over $\mathbb{Q}$, notice that this does not depend on the choice of the lift.
 
 If we suspend $g$ into a flow of $T^3$ by choosing an isotopy generating $g$, we can assign rotation vectors to such an isotopy as follows. The resulting flow $\phi: \mathbb{R}\times T^3\rightarrow T^3$ lifts to a flow $\widetilde \phi: \mathbb{R}\times \mathbb{R}^3\rightarrow \mathbb{R}^3$. The rotation vector of a point $p\in T^3$ is 
$$ \rho(p,\phi)=\lim_{t\rightarrow \infty}\frac{\widetilde \phi_t(\widetilde p)-\widetilde p}{t}, $$
where $\widetilde p$ is any lift of $p$. The homological direction is the projectivization $[\rho(p,\phi)]$, i.e. the equivalence class of the vector after identifying vectors that are strictly positive multiples of each other. We denote by $\rho(\phi)$ and $[\rho(\phi)]$ the set of rotation vectors and of homological directions of $\phi$.

In the following statement, we consider homology classes of global sections up to sign.
\begin{prop}\label{prop:sections}
 Let $h: \mathbb{A} \rightarrow \mathbb{A}$ be a smooth pseudorotation of the annulus that extends to a smooth pseudorotation of the torus. For any homology class $\sigma\in H_1(T^2,\mathbb{Z})$, the suspension flow of $h$ admits a global section $\Sigma$ such that $ \Sigma \cap \partial (T^2\times I)$ induces in each boundary torus a circle of homology class $\sigma$.
\end{prop}
{Notice that the suspension flow of a homeomorphism isotopic to the identity depends on the choice of isotopy generating the homeomorphism. However, the (smooth) orbit equivalence class of the flow is determined by the homeomorphism. We thus speak of ``the suspension" understood as this equivalence class.}
\begin{proof}
{Identify $T^2\cong \mathbb{R}^2/\mathbb{Z}^2$ and the annulus $\mathbb{A}$}  with the image of  $[0,1]\times [0,1/2]$ in the quotient $\mathbb{R}^2/\mathbb{Z}^2$. Let $H: T^2 \to T^2$
be an extension of $h$ to a pseudorotation of the torus. In particular, $H$ is isotopic to the identity, and for any lift of $H$ to $\mathbb{R}^2$, the rotation vector is of the form $(\alpha,n)$ for an irrational number $\alpha$ and $n\in \mathbb{Z}$. Fix a lift $\widetilde H$ such that $\rho(\widetilde H)=(\alpha,0)$, for an irrational number $\alpha$.

Let $\phi$ be the flow obtained by the suspension of $H$, defined in the 3-torus $T^3=\mathbb{R}^3/\mathbb{Z}^3=T^2\times S^1$ with coordinates $(x,y,t)$. 
Denote by $[\rho(\phi)]$ the set of homological directions of $\phi$, that is 
$$[\rho(\phi)]=\{ [\rho(p,\phi)]\enspace | \enspace p\in T^3\}.$$
 In the case of a suspension of a diffeomorphism of the torus whose rotation set is a single vector, the choice of $\widetilde H$ above implies that (we refer to \cite[Equation 2.2]{K1})
$$ [\rho(\phi)]=\{[(\alpha,0,1)]\}.$$
 Fried's Theorem \cite[Theorem D]{Fr} shows that there is a global section $\Sigma$ of $\phi$ inducing the Poincar\'e dual class $[-qdx+pdt]_{H^1(T^3;\mathbb{Z})}$ 
 as long as $-qdx+pdt$ is positive 
 when evaluated on the set of homological directions of $\phi$. 

{Since $[\rho(\phi)]=[(\alpha,0,1)]$}, given any pair of integers $(p,q)\neq (0,0)$, the cohomology class $[-qdx+pdt]$ does not vanish when evaluated at $[\rho(\phi)]$. If the value is positive, Fried's theorem implies that there is a global section $\Sigma$ of $\phi$ inducing the Poincar\'e dual class $[-qdx+pdt]_{H^1({T}^3;\mathbb{Z})}$; if the value is negative we use $[qdx-pdt]_{H^1(T^3;\mathbb{Z})}$. 

Now recall that the annulus $S^1\times [0,1/2]$ is invariant under $H$, and $\phi$ restricted to $S^1\times[0,1/2]\times S^1$ corresponds to a suspension of $h$. The surface $\Sigma$ restricts as an annulus-like section of the suspension of $h$ with associated class $[qdx-pdt]_{H^1(T^2\times I)}$. Since $(p,q)$ was arbitrary, we can find a global section with a prescribed homology class of $\Sigma\cap \partial (S^1\times [0,1/2] \times S^1)$.
\end{proof}

\begin{Remark}
 The previous proposition holds for the suspension of any pseudorotation of the annulus as long as Fried's theorem holds on compact manifolds with boundary. A different approach to try to prove this fact is using \cite[Proposition 5.1]{BCLP}. It follows from it that any pseudorotation is smoothly conjugate to a diffeomorphism that is arbitrarily $C^0$-close to a rigid rotation. If this diffeomorphism is connected to the rigid rotation by a $C^0$-small isotopy, a fact that is not implied by the $C^0$-closeness of the diffeomorphisms, then one can deduce the previous proposition for any pseudorotation of the annulus.
\end{Remark}

We now use the previous proposition to obtain a global section, up to cutting open $M$ along an invariant torus, for any aperiodic SHS Reeb vector field.

\begin{theorem}\label{thm:SHSannulus}
Let $X$ be the Reeb vector field of a stable Hamiltonian structure $(\lambda,\omega)$ without periodic orbits on a closed oriented 3-manifold $M$. If $f=\frac{d\lambda}{\omega}$ is non-constant, then given any torus $T\subset M$ contained in a regular level set of $f$, cutting $M$ along $T$ yields a manifold with boundary $\widetilde M$ which is diffeomorphic to $T^2\times I$ and in which the Reeb vector field admits an annulus-like section. Moreover, the vector field is orbit equivalent to the suspension of a pseudorotation of the annulus.
\end{theorem}

\begin{proof}
Take a structural decomposition $U=U_1\sqcup \ldots \sqcup U_k$ and $N= N_c \sqcup N_0$ of $(\lambda,\omega)$, as in Theorem~\ref{thm:struc}. By Taubes' theorem \cite{Taub}, $X$ is not the Reeb vector field of a contact sturcture and hence $U\neq \emptyset$. Let $V=\overline{U\setminus N}$ and, to simplify, choose a torus $T$ to be a connected component of $\partial V$. Then $T$ is contained in a regular level of $f$, is invariant under the Reeb flow and $T\subset \text{int}(U)$. Recall that each $U_i$ is an integrable region diffeomorphic to $T^2\times I$ where $(\lambda,\omega)$ can be assumed to be $T^2$-invariant.

Denote by $\widetilde{M}$ the compact manifold obtained by cutting open $M$ along $T$, whose boundary is given by two copies of $T$ that we denote by $T'$, $T''$. Notice that $T'$ and $T''$ both admit a neighborhood of the form $T^2\times I$ with a coordinate system for which $(\lambda,\omega)$ is $T^2$-invariant. The manifold $\widetilde M$ is also decomposed in the regions $V$ and $N$ (for which we use the same notation). Observe that in $\widetilde M$, $T'$ belongs either to $V$ or $N$, we assume that $T'\subset V$ and hence $T''\subset N$.

 Let us show that we can apply Proposition \ref{prop:BirkT2} to each connected component of $U$ (or $V$) and of $N$. We start by studying the integrable regions. Since the Reeb vector field has no periodic orbits, it has constant irrational slope in each integrable region $U_i$. In particular, the Reeb vector field in $U_i$ is orbit equivalent to the suspension of a rigid irrational rotation of the annulus that trivially extends to a smooth pseudorotation of the torus. 

 In each component of $N_0$ the Reeb vector field has a global section, and it is then orbit equivalent to the suspension of an area-preserving diffeomorphism of a compact surface without periodic orbits. This implies that each component of $N_0$ is diffeomorphic to $T^2\times I$ and that the Reeb vector field  is orbit equivalent to the suspension of a pseudorotation $h$ of the annulus, see \cite[page 20]{Re} or \cite[Section 5.3]{HT}. Observe that near the boundary of each connected component of  $N_0$ we have a family of invariant tori, implying that the pseudorotation of the annulus is a rigid irrational rotation near each boundary component of the annulus. Moreover, the rotation number of $h$ at each boundary has to be the same. This implies that $h$ can be extended to a pseudorotation of a torus, as in the hypothesis of Proposition~\ref{prop:sections}.
 
 In each connected component $N'$ of $N_c$, Lemma~\ref{lem:reebpseudo} implies that each component is diffeomorphic to $T^2\times I$ or $\mathbb{S}^1\times D^2$. Since $\widetilde M$ has two boundary components, we only have the first option. Moreover, the Reeb flow in $N'$ is orbit equivalent to the suspension of a pseudorotation $h$ of the annulus. Since the boundary components of $N'$ admit a neighborhood in $N'$ contained in $U$, the flow near $\partial N'$ is tangent to a foliation by tori and thus $h$ is a rigid rotation near the boundary. In particular, the diffeomorphism $h$ extends also in this case to a smooth pseudorotation of a torus, as in the hypothesis of Proposition~\ref{prop:sections}.

Consider now the regions $V_i=\overline{U_i\setminus N}\subset U_i$, that are each diffeomorphic to $T^2\times I$ and the restricted Reeb flow is orbit equivalent to the suspension of an irrational rotation of the annulus. We have thus shown that $\widetilde{M}$ is obtained by gluing along their boundary a finite number of connected domains each diffeomorphic to $T^2\times I$ where $X$ is orbit equivalent to the suspension of a pseudorotation (that always extends to a pseudorotation of the torus). In particular we have $\widetilde{M}\cong T^2\times I$.

\smallskip

{We proceed to prove that $X$ admits a global section in $\widetilde M$.} Consider the decomposition above $N\cup V$ and recall that $T'\subset V$. We construct the section starting in the connected component $V_1\subset V$ that contains $T'$. In $V_1$ the Reeb vector field is orbit equivalent to the suspension of an irrational rotation, that extends trivially to a rotation of a torus. By Proposition~\ref{prop:sections}, we can choose any non-trivial homology class $\sigma_1 \in H_1(\partial V_1; \mathbb{Z})$ and find an annulus-like section $\Sigma_1$ of $X|_{V_1}$ inducing the homology class $\sigma_1$ in each boundary component. The boundary component $T_1=\partial V_1\setminus T'$ is glued to a boundary torus of a connected component $N_1\subset N$, with $N_1\cong T^2\times I$. Denote this boundary torus by $T_1'\subset \partial N_1$. The surface $\Sigma_1$ induces on $T_1'$ a circle with homology class $\sigma_1'$ (the class $\sigma_1$ understood in $H_1(\partial N_1;\mathbb{Z})$). Applying Proposition~\ref{prop:sections}, there is an annulus-like section $\Sigma_1'$ of $X|_{N_1}$ in $N_1$ inducing circles with homology class $\sigma_1'$ in each boundary component of $N_1$. 

We now paste together $\Sigma_1$ and $\Sigma_1'$ to obtain a section of $X|_{V_1\cup N_1}$. Denote by $\gamma_1$ and $\gamma_1'$ the circles $\Sigma_1\cap T_1'$ and $\Sigma_1'\cap T_1'$, and choose an orientation for $T_1'$.  Up to a small $C^\infty$ perturbation of $\Sigma_1'$, we might assume that $\gamma_1$ and $\gamma_1'$ intersect transversely. Each curve comes equipped with an orientation, induced respectively by the orientations of $\Sigma_1$ and $\Sigma_1'$ inherited by the positive direction of the vector field $X$. Since $[\gamma_1]=[\gamma_1']$ in $H_1(T_1';\mathbb{Z})$, there is some oriented $2$-chain $C$ such that $\partial C=\gamma_1-\gamma_1'$. 

Assume that $\gamma_1\cap \gamma_1' \neq \emptyset$. The fact that $\gamma_1$ and $\gamma_1'$ intersect transversely implies that $C$ is given by a finite collection of disks. The disk interiors are two by two disjoint and the boundary of each disk has two corners. Notice that $T_1'$ is oriented and $C$ is oriented, hence we can distinguish two types of disks in $C$: the positive ones where both orientations coincide and the negative ones where the two orientations are opposite to each other. 
Let $D$ be one of the disks of negative sign. In a small neighborhood $\mathcal{V}$ of this disk, the curves $\gamma_1$ and $\gamma_1'$ intersect at exactly two points. Since the flow is transverse to $\gamma_1$ and $\gamma_1'$, assume first that in $D$ the orbits of the flow go from $\gamma_1$ to $\gamma_1'$. We can deform $\gamma_1$ inside $\mathcal{V}$ in the direction of the flow, so that $\gamma_1$ and $\gamma_1'$ no longer intersect in $\mathcal{V}$ and both are still transverse to the flow. If, contrary to the previous assumption, in $D$ the orbits of the flow go from $\gamma_1'$ to $\gamma_1$, we use the negative flow to deform $\gamma_1$. Doing this on all the negative disks,  
we can ensure that $\gamma_1$ and $\gamma_1'$ are disjoint. Since these curves are in the same homology class, they bound a cylinder. 

Let $\varphi_t$ denote the flow of $X$, we can find a continuous positive function $g: \gamma_1 \to \mathbb{R}$, which never zero, such $\varphi_{g(p)}(\gamma_1)=\gamma_1'$. Consider a small  $T^2$-invariant neighborhood $W=T^2\times [-\delta,\delta]$ of $T_1'$ in $\widetilde M$, where $X$ is just an irrational vector field of constant slope on each torus. Let $t$ be the coordinate on $[-\delta, \delta]$. We take $W$ so that $V_1\cap W= \{t\leq 0\}$ and $N_1\cap W=\{t\geq 0\}$. We might assume that $\Sigma_1\cap \{t=-\delta\}=\gamma_1$ and $\Sigma_1'\cap \{t=\delta\}=\gamma_1'$ (where we have abused notation by taking the translated curves in any of the tori). Now define surface $\Sigma$ which is equal to $\Sigma_1$ in $V_1\setminus W$, equal to $\Sigma_1'$ in $N_1\setminus W$ and such that in $W$ it is given by $(\varphi_{h(t,p)}(\gamma_1),t)$ for a continuous function $h(t,p)$ which is equal to $0$ near $t=-\delta$ and equal to $g(p)$ near $t=\delta$. Then $\Sigma$ is an annulus-like surface of section of the Reeb vector field in $V_1\cup N_1$. 

We can apply this argument iteratively by gluing the connected components of $V$ and $N$, showing that the Reeb vector field admits an annulus-like surface of section in  $\widetilde{M} \cong T^2\times I$ and hence that it is orbit equivalent to the suspension of a pseudorotation of the annulus.
\end{proof}

\subsection{Admissible torus bundles}

In this section we analyze aperiodic Reeb vector fields of SHS for which $f=\frac{d\lambda}{\omega}$ is not constant and deduce obstructions on the topology of the ambient $T^2$-bundle by applying Theorem~\ref{thm:SHSannulus}.

\begin{prop}\label{prop:bundles}
Let $(\lambda, \omega)$ be a SHS defining an aperiodic Reeb vector field such that $f=\frac{d\lambda}{\omega}$ is not constant. Then $M\cong T^3$ or $M$ is a positive hyperbolic torus bundle over $\mS^1$.
\end{prop}

\begin{proof}
Observe that if $f$ does not vanishes, then the Reeb vector field $X$ of $(\lambda, \omega)$ is the Reeb vector field of a contact structure and has periodic orbits by Taubes' theorem~\cite{Taub}. Hence $f$ vanishes somewhere and not everywhere.  Applying Theorem~\ref{thm:struc} we find a non-empty integrable region $U_1\cong T^2\times I$ where $(\lambda,\omega)$ is $T^2$-invariant. By cutting open the manifold $M$ along one of the torus fibers in the interior of $U_1$, we obtain a compact manifold with boundary $\widetilde M\cong T^2\times I$  and coordinates $(x,y,t)$ such that $(\lambda,\omega)$ is $T^2$-invariant near the boundary. The Reeb vector field $X$ defined by $(\lambda,\omega)$ is orbit equivalent to the suspension of a pseudorotation of the annulus by Theorem~\ref{thm:SHSannulus}, denote by $\alpha$ its irrational rotation number. Since the vector field $X$ corresponds to an irrational flow on each invariant torus near the boundary, near $t=0$ and $t=1$ it has irrational slope equal to $\alpha\in\mathbb{R}\setminus\mathbb{Q}$ on each torus. 

We now work near $t=0$, and write the form $\lambda$ and the vector field $X$ as in equations~\eqref{eq-lambdaonreg} and \eqref{eq-Xonregular}:
$$\lambda=g_1(t)dx+g_2(t)dy \qquad  X= a_1(t)\pp{}{x}+a_2(t)\pp{}{y}.$$
Then $\frac{a_2(t)}{a_1(t)}\equiv \alpha$. 

The kernel of $\lambda$ induces a linear foliation on the torus $t=0$ whose slope equal to $-\frac{g_1(0)}{g_2(0)}$. Assume that the function $-\frac{g_1(t)}{g_2(t)}$ is non-constant. Then, modulo changing the cutting torus inside $U_1$, we can assume that $-\frac{g_1(0)}{g_2(0)}$ is irrational and different from $\alpha$. Denote by  $\varphi:T^2 \to T^2 $
the gluing diffeomorphism such that $M$ is obtained by gluing the torus $t=1$ with the torus $t=0$ via $\varphi$. Such a diffeomorphism must preserve the irrational foliation by orbits of $X$, and the foliation given by the kernel of $\lambda$ restricted to $t=1$ must be sent to the foliation spanned by the kernel of $\lambda$ restricted to $t=0$. Let $\delta$ be the slope of the kernel of $\lambda$ at $t=1$, and $\beta=-\frac{g_1(0)}{g_2(0)}$ be the slope of the kernel of $\lambda$ along $t=0$. Write the map $\phi$ induced by $\varphi$ on $\mathbb{R}^2$, considered as the universal cover of $T^2$, in coordinates:
$$  \varphi(x,y)= (\varphi_1(x,y), \varphi_2(x,y)),$$
where $\varphi_1$ and $\varphi_2$ can be expressed in a unique way as 
\begin{align*}
\varphi_1(x,y)&=l_1(x,y)+p_1(x,y)\\
\varphi_2(x,y)&=l_2(x,y)+p_2(x,y)
\end{align*}
where $l_1,l_2:\mathbb{R}^2\to \mathbb{R}^2$ are linear functions and $p_1,p_2:\mathbb{R}^2\to \mathbb{R}^2$ are periodic functions. The fact that the kernel of $\lambda$ is preserved implies that there is some function $G:\mathbb{R}\to \mathbb{R}$ such that
$$ \varphi_2(x,y)-\beta \varphi_1(x,y)=G(y-\delta x).$$
We can now argue as in \cite[Lemma 4.4]{Re}: the function $p_2-\beta p_1$ is constant on the straight line of irrational slope $y=\delta x$ thus it is constant everywhere. This shows that $G(z)=az+b$ for some constants $a,b$ and so
$$ \varphi_2(x,y)-\beta \varphi_1(x,y)=a(y-\delta x)+b.$$
 The same argument applied to the foliation of slope $\alpha$ shows that
$$ \varphi_2(x,y)-\alpha \varphi_1(x,y)=c(y-\alpha x)+d, $$
for some constants $c,d$. Then 
$$c(y-\alpha x)+d+\alpha \varphi_1(x,y)= a(y-\delta x)+b+\beta \varphi_1(x,y),$$
and since $\beta\neq \alpha$, we obtain that $\varphi_1$ and $\varphi_2$ are linear functions. Up to a translation, the diffeomorphism $\varphi$ is given by a matrix $A \in SL(2,\mathbb{Z})$ which necessarily admits an eigenvector with irrational slope $\alpha$. This proves that $A$ is the identity or a hyperbolic matrix of $SL(2,\mathbb{Z})$.  Furthermore, notice that since the two boundary components are identified via $A$, the linear map should identify the foliation given by the orbits of $X$ in $T^2\times \{0\}$ with the foliation given by the orbits of $X$ in $T^2\times \{1\}$. Both foliations are linear of slope $(1,\alpha)$, and hence this vector is necessarily an eigenvector of $A$. However, if $A$ is negative hyperbolic, the foliation glues in a non-coorientable way, which contradicts the existence of $X$ in the first place. This shows that $M$ must be a positive hyperbolic bundle.

We are left with the case in which the function $-\frac{g_1(t)}{g_2(t)}$ is constant. We can assume that we chose the invariant domain $U_1$ such that $f\neq 0$ there. This implies that $d\lambda \neq 0$ on $U_1$, and it has the form
$$ d\lambda= g_1'(t) dt\wedge dx + g_2'(t) dt\wedge dy. $$
Since $\iota_Xd\lambda=0$, it follows that $\frac{g_1'(t)}{g_2'(t)}\equiv -\alpha$. Using that $\frac{g_1(t)}{g_2(t)}$ is constant we deduce that $\frac{g_1(t)}{g_2(t)}=\frac{g_1'(t)}{g_2'(t)}=-{\alpha}$. In other words, we have shown that the slope of the foliation induced by $\ker \lambda$ can be assumed to be irrational on the boundary of $\widetilde M$. We are now in the same situation as in the previous case, finishing the proof.
\end{proof}

\subsection{Proof of Theorem~\ref{thm:main1}}\label{ss:Tisch}

Let $M$ be a closed 3-dimensional manifold equipped with a stable Hamiltonian structure $(\lambda, \omega)$ whose Reeb vector field $X$ is aperiodic. This implies that $M$ is a torus bundle over $\mS^1$. Denote by $f$ the function $\frac{d\lambda}{\omega}$. This function necessarily vanishes somewhere since otherwise, the 1-form $\lambda$ is a positive or negative contact form and the Reeb vector field $X$ admits a periodic orbit by  Taubes' theorem 
\cite{Taub}.  

\textbf{First case: $f\equiv 0$.} Assume first that the function $f$ is constant and equal to $0$. In this case, the one form $\lambda$ is closed, and by Tischler's theorem \cite{Tis} $M$ is a fibration over $\mS^1$ such that each fiber is a global section of $X$. Then the fiber is necessarily a torus. Hence $X$ is orbit equivalent to the suspension of a symplectomorphism of the torus $\varphi:T^2\to T^2$, whose mapping class corresponds to a matrix $A\in SL(2,\mathbb{R})$.
The Lefschetz number of $\varphi$ is given by 
\begin{equation*}
\Gamma_{\varphi}=\sum_{i=0}^2 (-1)^i\operatorname{tr}(\varphi_*H_i(X,\mathbb{Q}))
								=2- \operatorname{tr}(A).
\end{equation*}
By hypothesis $\varphi$ admits no periodic points, which implies that $\Gamma_{\varphi}=0$ and hence $\operatorname{tr}(A)=2$. We deduce that $M$ is either $T^3$ or a positive parabolic torus bundle, and $X$ admits a global section. 

\textbf{Second case: $f\not \equiv 0$ and vanishes somewhere.} In this case Proposition~\ref{prop:bundles} implies that $M$ is either $T^3$ or a positive hyperbolic torus bundle. Theorem~\ref{thm:SHSannulus} shows that there is some embedded invariant torus $T$ such that the compact manifold $\widetilde M$ obtained by cutting open $M$ along $T$ is diffeomorphic to $T^2\times I$ where $X$ admits some annulus-like global section $\Sigma$ and $X$ is orbit equivalent to the suspension of a pseudorotation of the annulus. 

If $M=T^3$, then $\Sigma$ defines in $M$ an immersed surface with boundary $S$, embedded in the interior, and whose boundary is given by two circles defining the same homology class in $H_1(T^2;\mathbb{Z})$. We might argue exactly as in the proof of Theorem~\ref{thm:SHSannulus} to deform a bit $S$ so that it glues smoothly along $T$ and yields a global section of $X$ diffeomorphic to a torus.

If $M$ is a (positive) hyperbolic torus bundle, the missing part is to prove that $X$  does not admit a global section. Assume that there is such surface $S$, that we assume connected. Then after cutting along an invariant torus $T$, we obtain  a global section $\overline S$ of $X$ in $\widetilde M \cong T^2\times I$. Since $X$ in $\widetilde{M}$ admits an annulus-like global section, $\overline{S}$ has to be an annulus and it induces in the boundary of $\widetilde M$ circles with non-trivial homology class. 
The fact that $\overline S$ defines a smooth closed surface in $M$ after the identification of the two boundary components of $\widetilde M$ implies that such identification preserves some integer homology class. In other words, the matrix $A\in SL(2,\mathbb{Z})$ of the mapping class of $\varphi$ admits an eigenvector with integer coordinates, leading to a contradiction with the fact that $A$ is hyperbolic. We deduce that $X$ does not admit a global section, even if after cutting along $T$ it is orbit equivalent to the suspension of a pseudorotation of the annulus. 

Observe now that the rotation number of the annulus defines the slope of $X$ in appropriate coordinates. Arguing as in the proof of Proposition~\ref{prop:bundles}, we know that the gluing diffeomorphism between the two components of $\partial  \widetilde M$ is induced by the matrix $A$. Once we have our well-chosen generators of $H_1(T;\mathbb{Z})$, the irrational slope of $X$ is necessarily the same as the slope of some eigenvector of $A$. Writing the matrix $A$ in the form
\begin{equation*}
A=\begin{pmatrix}
a & b \\
c & d
\end{pmatrix}
\end{equation*}
it follows that the slope $\alpha$ of an eigenvector satisfies $b\alpha^2+(a-d)\alpha-c=0$, a quadratic equation. We conclude that $\alpha$ is a quadratic irrational number. 

Finally, when $(\lambda,\omega)$ is  real analytic, it was shown in \cite[Section 4]{Re} that $M$ is either $T^3$ or a parabolic $T^2$-bundle over $\mS^1$. Combining it with our previous discussion, we deduce that $M\cong T^3$ and only case (1) occurs. This finishes the proof of Theorem~\ref{thm:main1}.

\begin{Remark}\label{rem: pseudorotation}
 The fact that the rotation number is quadratic implies that it is not of Liouville type. In this case, no example is known of a smooth pseudorotation that is not conjugated to a rigid irrational rotation. It is an open question whether non-trivial pseudorotations with Diophantine (i.e. non Liouville) rotation number exist \cite{FK}.
\end{Remark}

\subsection{Aperiodic examples}
A natural question that arises is whether there is actually an aperiodic example in each case not forbidden by Theorem~\ref{thm:main1}. In $T^3$, an example is given by the suspension of the time one map of any linear vector field on $T^2$ with irrational slope. In a positive parabolic bundle, an example is given by the suspension of the map
\begin{align*}
\phi: T^2 &\longrightarrow T^2\\
	 (x,y) &\longmapsto (x+ny,y+\alpha),
\end{align*} 
where $\alpha$ is an irrational number and $n\neq 0$ is an integer that determines the homeomorphism type of the resulting (positive) parabolic torus bundle. 

\medskip

However, we do not know if an example exists for positive hyperbolic torus bundles. The simplest candidate can be constructed as follows. Given a hyperbolic matrix $A\in \operatorname{SL}(2,\mathbb{Z})$ with positive eigenvalues, choose an eigenvector $(1,\alpha)$ 
 ($\alpha \in \mathbb{R}\setminus \mathbb{Q}$) and consider a suspension of an irrational rotation of the annulus with rotation number $\alpha$. Using that $(1,\alpha)$ is an eigenvector and choosing a good parametrization of the suspended flow, one obtains a flow that descends to the quotient $M_A$ obtained by identifying the two boundaries of $T^2\times I$ using $A$. Let us call $\mathcal{F}_\alpha$ the 1-dimensional foliation of $M_A$ given by the orbits of such a flow.  Observe that if we do this construction for a matrix $A$ with negative eigenvalues, the foliation $\mathcal{F}_\alpha$ is not coorientable.
 
 The question at this point is whether this $\mathcal{F}_\alpha$  can be defined by the orbits of a Reeb flow of a SHS. One easily sees that the foliation lies in the kernel of a closed 2-form, so the only relevant point is whether a stabilizing 1-form exists, i.e., whether the foliation is {\it geodesible} \cite{Sul}. It turns out that this is never the case, as shown in \cite[Remark 2.1]{Ac} using the basic cohomology of the foliation. Hence, if a SHS with an aperiodic Reeb vector field exists in some hyperbolic torus bundle, after cutting open along an invariant torus the flow would have to be orbit equivalent to the suspension of a non-trivial pseudorotation of the annulus with quadratic rotation number. As we mentioned in Remark~\ref{rem: pseudorotation}, it is not clear whether such pseudorotations exist.

\section{SHS with finitely many periodic orbits}\label{s:finite}

In this section, we are interested in understanding when a contact non-degenerate SHS admits, or not, infinitely many periodic orbits.

\subsection{Symplectomorphism of surfaces aperiodic at the boundary}

Let $\Sigma$ be a compact surface. A diffeomorphism $\phi:\Sigma \to \Sigma$ is called a symplectomorphism if there exists an area form $\omega$ in $\Sigma$ such that $\phi^*\omega=\omega$.  To simplify the statement of the following theorem, we require that $\phi$ is aperiodic along the boundary. However, the same theorem holds if we require only that there is a non-empty set of connected components of the boundary that is $\phi$-invariant and where $\phi$ has no periodic points.

\begin{theorem} \label{thm:surfirra}
Let $\Sigma$ be a surface with non-empty boundary. Let $\phi:\Sigma \to \Sigma$ be a symplectomorphism without periodic points along the boundary. Then $\phi$ admits periodic points of arbitrarily large period in the interior of $\Sigma$ unless either $\Sigma\cong D^2$ and $\phi$ is a pseudorotation of the disk, or $\Sigma \cong \mathbb{A}$ and $\phi$ is a pseudorotation of the annulus.
\end{theorem}

\begin{proof}
Let $\Sigma=\Sigma_g^b$ be a surface of genus $g$ with $b>0$ boundary components. Up to considering an iterate of $\phi$, we can assume that each boundary component is preserved. Then $\phi$ restricts on each boundary component to a diffeomorphism of the circle that is conjugated to an irrational rotation (we will only need this property in one of the boundary components). 

We can blow-down each boundary component of $\Sigma_g^b$ to a point. We obtain a closed surface of genus $g$ on which $\phi$ induces a homeomorphism that preserves a measure that is positive on non-empty open sets  
$$ \varphi: \Sigma_g \longrightarrow \Sigma_g. $$
Then every point of $\Sigma_g$ is non-wandering under $\varphi$.
Observe that $\varphi$ has at least $b$ fixed points, that correspond to the connected components of $\partial \Sigma_g^b$. Choose one of them that we call $z\in \Sigma_g$. For the proof, we consider three cases: $g=0$,  $g=1$ and $g\geq 2$.

\smallskip

 We start with the case $g=0$, thus $\Sigma_g$ is a sphere. Theorem~\ref{thm:Fk} shows that either $\varphi$ is a pseudorotation (if the rotation set is a single point), or $\varphi$ has infinitely many periodic points of arbitrarily large periods. We deduce that $\phi$ has infinitely many periodic points of arbitrarily large period or $\Sigma_g^b$ is either a disk $D^2$ or an annulus $\mathbb{A}$, and in both cases $\phi$ is a pseudorotation. Observe that if a positive iterate of a homeomorphism is a pseudorotation, the same holds for the homeomorphism itself.  

\smallskip

Assume now that the surface is a torus $\Sigma_g=\Sigma_1=T^2$. Consider $\mathbb{R}^2$ as the universal cover of $T^2$, and
$$ \widetilde \varphi: \widetilde \Sigma_1  \longrightarrow \widetilde  \Sigma_1,$$
a lift of $\varphi$ such that $\widetilde \varphi(\widetilde z)=\widetilde z$ for  some lift $\widetilde z$ of $z$. The homeomorphism $\widetilde \varphi$ has no wandering points. Decompose $\widetilde \varphi$ as
 $\widetilde \varphi(x,y)= A(x,y) + g(x,y)$,
where $A\in SL(2,\mathbb{Z})$ and $g$ is $\mathbb{Z}$-periodic in both coordinates. 

If $A$ is a hyperbolic element of $SL(2,\mathbb{Z})$, it is well known that the map $\varphi$ necessarily admits infinitely many periodic points of arbitrarily large period \cite{F0, Ha}. Thus we assume that $A$ is either parabolic, elliptic, or the identity.

Next we identify $\mathbb{R}^2$ with the open unit  disk $D^2$ via the map $\rho(x,y)=\frac{(x,y)}{||(x,y)||+1}$. We can construct a continuous extension of $\widehat \varphi= \rho\circ \widetilde \varphi \circ \rho^{-1}$ to $ \overline{D^2}$ as follows: any sequence of points of the form $(u_n,v_n)$ such that $ (u_n,v_n) \rightarrow (u,v) \in \partial D^2$ satisfies 
$$\lim_{n \to \infty} \rho \circ \widetilde \varphi \circ \rho^{-1} (u_n,v_n) = \frac{A(u,v)}{||A(u,v)||}.$$
This allows us to extend the map $\widehat \varphi$ to a map of $\overline{D^2}$.  

We now consider the different types of $A\in SL(2,\mathbb{Z})$. If $A$ is of finite order, i.e. an elliptic element of $SL(2,\mathbb{Z})$, then $\widehat \varphi$ corresponds to a rational rotation along $\partial D^2$. If $A$ is parabolic, we can find an eigenvector of $A$ which induces a fixed point of $\widehat \varphi$ along the boundary. Finally, if $A$ is the identity, then $\widehat \varphi$ is just the identity on the boundary $\partial D^2$.  In the three cases, observe that $\widehat \varphi$ admits either a fixed point or a periodic point along the boundary of the disk. This implies that the rotation number of $\widehat \varphi$ along the boundary is rational.

Let $\hat z=\rho(\widetilde z)\in D^2$. We blow up $\hat z$, obtaining a homeomorphism with no wandering points of the closed annulus $\mathbb{A}$. The boundary component obtained by blowing up $\hat z$ is invariant by this homeomorphism and the dynamics restricted to it is not necessarily a rigid rotation, but the rotation number is irrational. On the other boundary component, the rotation number is rational. Hence, these numbers are different and we deduce that the rotation set of the homeomorphism of the annulus has at least two points (see \textsection\ref{ssec-surfacedyn} for the definition of the rotation set). Thus, by Theorem~\ref{thm:Fk} there is a compact subset $K$ of $D^2$ with periodic points of arbitrarily large period. The compactness of $K$ implies that each point in $\Sigma_g$ has only finitely many preimage points in $K$ by the projection from the universal cover, so we deduce that the projection of $K$ in $\Sigma_g$ contains periodic points of $\varphi$ of arbitrarily large period. It follows that $\phi:\Sigma_1^b \to \Sigma_1^b$ also admits infinitely many periodic points of arbitrarily large periods. 

\smallskip

It only remains to analyze the case where the surface is $\Sigma_g$ with $g \geq 2$. By the Nielsen-Thurston decomposition \cite{T, FLP}, we know that there is a positive integer $q$ such that one of the following holds:
\begin{itemize}
\item[-] $\varphi^q$ is isotopic to the identity,
\item[-] $\varphi^q$ is isotopic to a Dehn twist,
\item[-] the decomposition of $\varphi$ has at least one pseudo-Anosov component.
\end{itemize} 
If $\varphi$ has a pseudo-Anosov component, it is well-known that $\varphi$ admits infinitely many periodic points of arbitrarily large periods \cite{Ha}. Otherwise, denote $\varphi^q$ by $\Phi$. The universal cover of $\Sigma_g$ is the Poincar\'e disk  $\widetilde \Sigma_g \cong D^2$ and consider a lift of $\Phi$
$$ \widetilde \Phi: D^2 \longrightarrow D^2 $$
such that $\widetilde \Phi(\widetilde z)=\widetilde z$, where $\widetilde z$ is one lift of the fixed point $z$. Such homeomorphism admits an extension to the boundary of the disk $S_\infty=\partial D^2$ (see e.g. \cite[Corollary 1.2]{T}). 

\begin{lemma}\label{lemma: rational}
    The rotation number of $\widetilde{\Phi}|_{S_\infty}$ is a rational number.
\end{lemma}

\begin{proof}
We will show that $\widetilde{\Phi}|_{S_\infty}$ necessarily has a periodic point, which implies that its rotation number is rational. Let $\Phi'$ be a diffeomorphism isotopic to $\Phi$ which is either the identity or a Dehn twist. We can lift the isotopy, obtaining a homeomorphism $\widetilde \Phi'$ isotopic to $\widetilde \Phi$. Abusing notation, we also denote by $\widetilde \Phi'$ and $\widetilde \Phi$ the extended homeomorphisms to the closed disk. 

If $\Phi'$ is the identity then $\widetilde \Phi'$ coincides with a deck transformation of $D^2$, which is given by a hyperbolic translation. This readily implies that $\widetilde \Phi'$ admits two fixed points (the endpoints of the translation axis) along $S_\infty$. Since $\widetilde \Phi'|_{S_\infty}=\widetilde \Phi|_{S_\infty}$, we deduce that $\widetilde \Phi$ also admits two fixed points along the boundary of the disk. This proves Lemma~\ref{lemma: rational} in the first case.

Otherwise, $\Phi'$ is a Dehn twist and we can find some simple closed curve $\gamma$  non-trivial in homology that is preserved by $\Phi'$. The lift of $\gamma$ to the universal cover is an infinite family of disjoint open segments, and the closure of each intersects $S_\infty$ at two points. Fix one of such lifts $\widetilde \gamma_0$, with boundary points $p_1,p_2 \in S_\infty$. Since $\gamma$ is preserved by $\Phi'$,  $\widetilde \Phi'(\widetilde \gamma_0)$ is another segment $\widetilde \gamma_1$ with boundary points $q_1, q_2 \in S_\infty$ (that projects into $\gamma$). Both segments $\widetilde \gamma_0$ and $\widetilde \gamma_1$ (with their boundary points included) are either equal or disjoint. If $\widetilde{\gamma}_0=\widetilde{\gamma}_1$, then $(\widetilde{\Phi'})^2(p_1)=(\widetilde{\Phi})^2(p_1)=p_1$ and Lemma~\ref{lemma: rational} is proved. If $\widetilde{\gamma}_0\neq \widetilde{\gamma}_1$, assume without loss of generality that $\widetilde \Phi'(p_1)=q_1$ and $\widetilde \Phi'(p_2)=q_2$. 

 Let us show that $\widetilde \Phi'$ admits a periodic point at the boundary. Let $J_0$ be the closed interval in $S_\infty$ whose boundary is $p_1$ and $p_2$ and such that $q_1,q_2$ are not in the interior of $J_0$. Such an interval exists because the curves $\widetilde \gamma_0$ and $\widetilde \gamma_1$ are disjoint. Denote by $J_1$ the interval $\widetilde \Phi'(J_0)$ which satisfies  $\partial J_1=\{q_1,q_2\}$. Two cases can occur. In the first case, we have $J_0\subset J_1$, and we can apply Brouwer's fixed point theorem to $(\widetilde \Phi')^{-1}$, deducing that $\widetilde \Phi'$ admits a fixed point. In the second case, $J_0$ and $J_1$ are disjoint. Consider the interval $H=(\widetilde \Phi')^2(J_0)$ and we claim either $H\subseteq J_0$ or $J_0\subset H$.  To see this, note first that $H$ is the image of $J_1$ by $\widetilde \Phi'$ and hence lies in the complementary of the interior of $J_1$, because $J_1=\widetilde \Phi'(J_0)$ and we are assuming that $J_0$ and $J_1$ are disjoint along their interior.  Furthermore, the boundary points of $H$ lie in the boundary of some curve $(\widetilde \Phi')^2(\widetilde \gamma)$ which is another lift of $\gamma$ and so is either equal to or disjoint with $\widetilde \gamma$. This shows that the boundary points of $H$ are either both contained in $J_0$ or both in $S_\infty\setminus J_0$, since otherwise the curves $\widetilde \gamma$ and $(\widetilde \Phi')^2(\widetilde \gamma)$ would be different but with non-trivial intersection. 
 
 We conclude that if the boundary points of $H$ lie in $J_0$ then $H\subseteq J_0$, and if they lie outside of $J_0$ then $J_0\subset H$. We can apply Brouwer's fixed point theorem to $(\widetilde \Phi')^2$ or to $(\widetilde \Phi')^{-2}$ to conclude that $\widetilde \Phi'$ admits a periodic point of period two. Since $\widetilde \Phi'|_{S_\infty}=\widetilde \Phi|_{S_\infty}$, we deduce that $\widetilde\Phi $ always admits a periodic point at the boundary.
 \end{proof}

We can now finish the proof of Theorem~\ref{thm:surfirra} as in the case of the torus. The lift $\widetilde \Phi$ fixes a point $\widetilde z$ with an irrational rotation number, and since it admits a periodic point along $S_{\infty}$, it has a rational rotation number along $S_\infty$. We can blow up the fixed point $\widetilde{z}$ to obtain a homeomorphism with no wandering points of the closed annulus with different rotation numbers at each boundary component, and conclude as before that $\varphi$, and hence $\phi$, admits periodic points of arbitrarily large period.
\end{proof}


\begin{Corollary}
Let $(\lambda, \omega)$ be a stable Hamiltonian structure ({possibly} degenerate), and fix some structural decomposition. If $N_0\neq \emptyset$ and there is a connected component of $N_0$ that is not a solid torus or a thickened torus ($T^2\times I$), then the Reeb vector field admits infinitely many periodic orbits.
\end{Corollary}

\begin{proof}
Assume by contradiction, that there is a connected component of $N_0$ that is neither a solid torus, nor $T^2\times I$ and that the Reeb vector field has finitely many periodic orbits. Then, in every invariant torus in $U$ the linear vector field has irrational slope. The Reeb flow in $N_0$ is orbit equivalent to the suspension of a symplectomorphism of a surface with boundary, which by hypothesis is neither a disk nor a torus. It follows from Theorem~\ref{thm:surfirra} that the Reeb vector field admits infinitely many periodic orbits.
\end{proof}

\subsection{Proof of Theorem~\ref{thm:main2}}

Let $M$ be a closed manifold endowed with  a contact non-degenerate SHS  $(\lambda,\omega)$, i.e. the Reeb vector field is non-degenerate in $N_c$ for some structural decomposition $U$, $N_0$, $N_c$. Assume that the Reeb vector field $X$ is not aperiodic. We separate the proof depending on whether $U$ is empty or not. For each case,  we first prove that the only contact non-degenerate cases with at least one and finitely many periodic orbits are the two options listed in Theorem~\ref{thm:main2}. We then explain why 
the statement holds in a neighborhood of such a SHS.

{If $U=\emptyset$, then either $M=N_c$ or $M=N_0$. If $M=N_c$,} then the Reeb vector field $X$ is, up to reparametrization, a non-degenerate Reeb vector field of the contact form $\lambda$. 
It follows from \cite[Theorem 1.2]{CDR} that the Reeb vector field of $(\lambda,  \omega)$ has infinitely many periodic orbits unless either $M$ is the sphere or $M$ is a lens space and there are exactly two periodic orbits that are core circles of a genus one Heegaard splitting of $M$. 

\smallskip
Consider now the case in which $M=N_0$. Then the vector field $X$ is orbit equivalent to the suspension of a symplectomorphism of a closed surface, so we deduce that if $X$ has finitely many periodic orbits, it is orbit equivalent to the suspension of a symplectomorphism of a surface with finitely many periodic points.

\smallskip

The last case to analyze is when $U$ is non-empty. Assume that $(\lambda,\omega)$ is such that its Reeb vector field has finitely many periodic orbits, and at least one. For each connected component of $N_0$ we know that the Reeb vector field $X$ is orbit equivalent to a suspension of symplectomorphism of a surface with boundary $\Sigma_g^b$. Then in each boundary component of $N_0$ the Reeb vector field is orbit equivalent to an irrational linear flow, so in the boundary of $\Sigma_g^b$ the corresponding symplectomorphism has no periodic points. By Theorem~\ref{thm:surfirra} we deduce that each component of $N_0$ is diffeomorphic to a solid torus or to $T^2\times I$ and $X$ restricted to each connected component of $N_0$ is orbit equivalent to the suspension of a pseudorotation.  

Now, in each connected component of $N_c$, the vector field $X$ is a non-degenerate Reeb vector field of the contact form $\lambda$, which is $T^2$-invariant near the boundary. 
Assuming that $X$ has finitely many periodic orbits, Lemma~\ref{lem:reebpseudo} implies that each connected component of $N_c$ is diffeomorphic to a solid torus or to $T^2\times I$, and that $X$ restricted to each connected component of $N_0$ is orbit equivalent to a suspension of a pseudorotation. 

The connected components of $N$ are glued along their boundary to boundary components of $V=\overline{U\setminus N}$. Each connected component of $V$ is diffeomorphic to $T^2\times I$ and, since $X$ has finitely many periodic orbits,  $X$ restricted to one connected component of $V$ is orbit equivalent to the suspension of an irrational rotation of an annulus. 

In conclusion, the manifold $M$ is obtained by gluing a finite number of copies of $\mS^1\times D^2$ and $T^2\times I$ along their boundaries. Since $M$ is assumed to be connected, we deduce that there are either two or zero copies of $\mS^1\times D^2$. In the first case, the vector field $X$ has no periodic orbits, since in each domain $T^2\times I$ it is orbit equivalent to the suspension of a pseudorotation of the annulus. Since we assumed that $X$ has at least one periodic orbit, we deduce that exactly two connected components in $N$ are diffeomorphic to  $\mS^1\times D^2$. After gluing together all the components that are diffeomorphic to $T^2\times I$ iteratively to one of the solid tori, we obtain a decomposition of $M$ into two solid tori $V_1, V_2$ that share a common boundary. This shows that $M$ is either $S^2\times S^1$ and the vector field is conjugate to the suspension of a symplectomorphism of the sphere with exactly two periodic points, or a sphere or a lens space and the periodic orbits of $X$ are core circles of $V_1$ and $V_2$ which define a genus one Heegaard splitting of $M$. 

\smallskip

This shows that the conclusions of Theorem~\ref{thm:main2} for the Reeb vector field of a contact non-degenerate SHS. We now prove that the same holds for a SHS with a periodic Reeb orbit and that is $C^2$-close to a contact non-degenerate SHS $(\lambda,\omega)$.

Denoting as before $U,N_c, N_0$ a structural decomposition of $(\lambda,\omega)$, assume first that $U=\emptyset$ and $M=N_c$. The Reeb vector field of any $C^2$-close SHS $(\widetilde \lambda, \widetilde \omega)$ will be the Reeb vector field of the contact form $\widetilde \lambda$ which is $C^2$-close to $\lambda$. Since \cite[Theorem 1.2]{CDR} is valid in a $C^2$-neighborhood of the contact form $\lambda$ we deduce that the conclusions of Theorem \ref{thm:main2} hold for $(\widetilde \lambda, \widetilde \omega)$. 

Assume now that $U=\emptyset$ and $M=N_0$. Then the Reeb vector field admits a global section, and this is an open condition. Thus the Reeb vector field of any SHS $C^2$-close to $(\lambda,\omega)$ is orbit equivalent to the suspension of a symplectomorphism, and the proof of Theorem~\ref{thm:main2} applies directly in this case. 

The last case is when $U\neq \emptyset$. By \cite[Theorem 3.7]{CV1}, given a stable Hamiltonian structure $(\widetilde \lambda, \widetilde \omega)$ sufficiently $C^2$-close to $(\lambda,\omega)$, there will be a $T^2$-invariant integrable region $K_i$ inside each connected component $U_i$ of $U$ of almost full measure in $U$. Let $K$ be the disjoint union of the $K_i$. In particular we can assume that $N_0$ and $N_c$ still have a neighborhood of each of their boundary components contained in $K$, and we can redefine the domains $N_0, N_c$, so that their boundary is an invariant torus of the Reeb vector field of $(\widetilde \lambda, \widetilde \omega)$ in $K$. Assume that the vector field $\widetilde X$ of $(\widetilde \lambda, \widetilde \omega)$ has finitely many periodic orbits and at least one periodic orbit. Then the slope of $\widetilde X$ has constant irrational slope in each $K_i$. The complement of  $K$, which is a slightly shrunk version of $U$, is diffeomorphic to $N_0\sqcup N_c$. In $N_0$, the Reeb vector field still admits a global section, since there it is $C^1$-close to the Reeb vector field of $(\lambda,\omega)$, which is orbit equivalent to a suspension flow. In the contact region $N_c$, the Reeb vector field of $(\widetilde \lambda, \widetilde \omega)$ is the Reeb vector field of the contact form $\widetilde \lambda$, which is $C^2$-close to $\lambda$. Since a neighborhood of the boundary of $N_c$ is contained in the interior of $K$, we are in the setting of Lemma~\ref{lem:reebpseudo} except that the contact form is not necessarily non-degenerate. However, the lemma still applies to (contact) Reeb vector fields in a $C^1$-small neighborhood of non-degenerate ones, since so does \cite[Theorem 1.2]{CDR}. Then each argument of our proof for the contact non-degenerate case above applies as well to the Reeb vector field of $(\widetilde \lambda, \widetilde \omega)$. This concludes the proof of Theorem~\ref{thm:main2}. 

\begin{Remark}\label{rem:degeneratecase}
Notice that if \cite[Theorem 1.2]{CDR} and hence Lemma~\ref{lem:reebpseudo} can be proven without the non-degeneracy hypothesis, then our proof applies and Theorem~\ref{thm:main2} holds without the contact non-degeneracy hypothesis. 
\end{Remark}

\section{Broken books and Birkhoff sections for SHS}\label{s:SHSsections}

The aim of this section and the next one is to generalize the main results of \cite{CDR} and \cite{CDHR, CM}, from Reeb vector fields of a contact form to Reeb vector fields of a SHS. In particular, we give a proof of Theorem~\ref{thm:mainBirk}. We start by recalling the following definitions of subsets of SHS:
\begin{enumerate}
\item the set $\mathcal{B}$ of contact non-degenerate SHS such that, there is a structural decomposition with a contact region where the Reeb vector field is non-degenerate and in each connected component of the integrable region $U$ the slopes of the Reeb vector field are non-constant (see \textsection\ref{ss:SHS} for more details).
    \item the set $\mathcal{B}_s\subset \mathcal{B}$ of strongly contact non-degenerate SHS such that, there is a structural decomposition with a contact region where the Reeb vector field is strongly non-degenerate and in each connected component of the integrable region $U$ the slopes of the Reeb vector field are non-constant.
\end{enumerate}

Broken book decompositions, introduced in \cite{CDR} (consult Definitions~\ref{defn:BBD} and \ref{defn:BBD2}), provide a strong tool for studying the dynamics of a vector field. In particular, they are the starting point to construct Birkhoff sections under generic hypotheses (see \cite{CDHR} and \cite{CM}). We prove the existence of a supporting broken book decomposition for  Reeb vector fields of a SHS in $\mathcal{B}$.

We now recall the definition of Birkhoff section, which we mentioned in the introduction, and introduce the notion of $\partial$-strongness.

\begin{defi}
    A Birkhoff section of a flow $X$ on a compact 3-manifold $M$ is an immersed surface $i:S \longrightarrow M$ satisfying that
    \begin{itemize}
        \item[-] it is embedded and transverse to $X$ in the interior,
        \item[-] the boundary is tangent to $X$,
        \item[-] it intersects any positive trajectory of $X$ in bounded time.
     \end{itemize}
\end{defi}

If we remove the last property, the surface is called a transverse surface. Given a periodic orbit $\gamma$ of $X$ we denote by $\Sigma_\gamma$ the unit normal bundle~$(TM_\gamma/T\gamma)/\R_+$ to~$\gamma$ and by~$M_\gamma$ the normal blow-up of~$M$ along~$\gamma$, that is the manifold~$(M\setminus\gamma)\cup \Sigma_\gamma$.
The vector field~$X$ extends to a vector field~$X_\gamma$ on the torus~$\Sigma_\gamma$ and hence to a vector field on~$M_\gamma$ tangent to the boundary. Observe that $X_\gamma$ restricted to the interior of $M_\gamma$ coincides with $X$ in $M\setminus \gamma$.
We abuse notation and still denote this extension by~$X$.
If $S$ is a transverse surface in $M$ with~$\gamma \subset \partial S$, we denote by~$\partial_\gamma S$ the extension to~$\Sigma_\gamma$ of the boundary components of $S$ that cover $\gamma$.

\begin{defi}\label{defn-strong}
A transverse surface $S$ is {\it $\partial$-strong} if, for every boundary orbit $\gamma$ of~$S$, the extension~$\partial_\gamma S$ is a collection of embedded curves in~$\Sigma_\gamma$ that are transverse to the vector field~$X_\gamma$.
If $S$ is a Birkhoff section, $S$ is {\it $\partial$-strong} if moreover $\partial_\gamma S$ defines a section of~$X_\gamma$. 
\end{defi}

The rough idea of the existence of broken books is the following. First, we assume that the Reeb vector field of a SHS is neither a suspension nor a Reeb vector field of a contact form. In $N_0$ we have global sections whose boundary is contained in $\partial N_0$, while in $N_+$ and $N_-$ we have a broken book decomposition whose pages might have boundary components contained in $\partial N_+$ and $\partial N_-$ respectively. We want to paste the boundaries of these two types of transverse surfaces along the integrable region, and this is done with  {\it helix boxes} of the form $T^2\times I$ contained in the region $U$, as explained in \textsection\ref{sec:gluing}.

We start with the construction of the helix boxes 
in \textsection\ref{sec:gluing}. Then in \textsection\ref{sec:bbook} we prove Theorem~\ref{thm:mainBirk}. 

\subsection{Birkhoff sections in $T^2\times I$}\label{sec:gluing}

{We proceed to the construction of Birkhoff sections (see \textsection\ref{sec:intro}) with prescribed boundary behavior of certain vector fields in $T^2\times I$.} We call $T^2\times I$ endowed with the  Birkhoff section a helix box. Consider $T^2\times I$ with a non-singular vector field $X$ that is tangent to the boundary. As in \textsection\ref{sec:prelim}, we say that $X$ is $T^2$-invariant if it is tangent to and conjugate to a linear flow in each torus fiber of $T^2\times I$ (as in Equation~\eqref{eq-Xonregular}). It would be enough that $X$ is orbit equivalent to a linear flow on each fiber, but then up to reparametrization the flow is just $T^2$-invariant.

\begin{prop}\label{prop:helix}
Let $X$ be a vector field on $T^2\times I$ that is $T^2$-invariant, with a periodic orbit $\nu$ in $T^2\times \{t^*\}$ for some $t^*\in(0,1)$. Let $\gamma_0,\gamma_1$ be two connected closed curves respectively in $T^2 \times \{0\}$ and $T^2\times \{1\}$. Assume that:
\begin{enumerate}
    \item $\gamma_0\times I$ and $\gamma_1\times I$ are positively transverse to $X$ respectively in $T^2\times [0,t^*]$ and $T^2 \times [t^*,1]$;
    \item the homology classes $[\gamma_0]$ and $[\nu]$ generate $H_1(T^2;\mathbb{Z})$;
    \item  $[\gamma_0]\neq [\gamma_1]$ in $H_1(T^2;\mathbb{Z})$.
\end{enumerate}
Write $[\gamma_1]= p [\gamma_0] + q [\nu]$. Then $X$ admits a $\partial$-strong Birkhoff section $\Sigma$ with binding $\nu$ (covered $-q$ times, where $\nu$ is oriented by the flow) and such that $\Sigma\cap \{t=1\}=\gamma_1$ and $\Sigma \cap \{t=0\}$ is given by $p$ parallel copies of $\gamma_0$.
\end{prop}

{Since $\gamma_0$ and $\gamma_1$ are connected curves, the coefficient $q$ in the theorem cannot be equal to zero unless $(p,q)=(1,0)$ contradicting part (3) of the hypothesis.} If $[\gamma_0]=[\gamma_1]$, the vector field admits an annulus-like section without boundary components in $T^2\times (0,1)$, justifying hypothesis (3). In particular, we have that in the relation $[\gamma_1]= p [\gamma_0] + q [\nu]$ both coefficients are non-zero. 

\begin{proof}
{Notice that because $\gamma_0$ and $\gamma_1$ are connected, their homology classes are primitive.} In a small neighborhood $V=T^2\times [t^*-\varepsilon,t^*+\varepsilon]$ of $T^2 \times \{t^*\}$, the  vector field $X$ has slope close to the one at $T^2\times \{t^*
\}$, where $\nu$ is a periodic orbit of $X$. Hence the surfaces $\gamma_0\times I$ and $\gamma_1 \times I$ are  transverse to $X$ in $V$. We will construct a Birkhoff section in $V$ that intersects $T^2\times\{t^*-\varepsilon\}$ in $p$ parallel copies of the curve $\gamma_0 \times \{t^*-\varepsilon\}$ and intersects $T^2\times\{t^*+\varepsilon\}$ in the curve $\gamma_1\times \{t^*+\varepsilon\}$. In particular, this Birkhoff section can be extended to a Birkhoff section to all $T^2\times I$ trivially with $p$ parallel copies of the surface $\gamma_0\times [0,t^*-\varepsilon]$ and the surface $\gamma_1 \times [t^*+\varepsilon,1]$.

{Since the homology classes of $\gamma_0$ and $\nu$ are assumed to generate $H_1(T^2;\mathbb{Z})$, we call the direction of $\nu$ the vertical direction and the direction of $\gamma_0$ the horizontal direction.} In $V$ the vector field is almost constant, thus we simplify the situation by considering: $X$ to be a constant vector field parallel to the vertical direction. Finally, we identify $V$ with a cube $I^3$ with coordinates $(x,y,t)$. This cube identifies with $V$ under the relations $(0,y,t) \sim (1,y,t)$ and $(x,0,t) \sim (x,1,t)$ as 
$$\bigslant{I^3}{\sim} \cong V\cong T^2\times I.$$
{In the coordinates $(x,y,t)$, we set  $\nu$ to be the vertical} curve $\{x=1/2,t=1/2\}$ (in particular we have $t^*=1/2$ and the vector field is parallel to $\pp{}{y}$ in $V$ by our assumption above) and $\gamma_0$ the curve $\{y=0\}$ and $\{t=0\}$. 
The closed curve $\gamma_1$ in $T^2\times \{1\}$ can be assumed to be a linear constant slope curve up to isotopy. Recall that
$[\gamma_1]=p[\gamma_0]+ q[\nu]$,
for some integers $p,q\in \mathbb{Z}$. 

\begin{lemma}\label{lem:Birk}
Assume that $[\gamma_1]=p [\gamma_0] + q [\nu]$ is a primitive homology class in $H_1(T^2;\mathbb{Z})$. With the notation above, the vector field $X$ admits a Birkhoff section $\Sigma$ in $V$ whose boundary is $-q$ times the periodic orbit $\nu=\{x=1/2, t=1/2\}$ and such that $[\Sigma\cap \{t=0\}] =p [\gamma_0]$ and $\Sigma \cap \{t=1\}=\gamma_1$.
\end{lemma}

The idea of the proof is to construct suitable curves in the boundary of $I^3$ and then use them to span a surface whose boundary is $-q\nu$ and curves in $\partial I^3$, see Figure~\ref{Fig:Birk}. Finally, we quotient to obtain a Birkhoff surface in $V$.

\begin{proof}
Let us assume that $q>0$. Up to a translation, we might assume that the curve $\gamma_1$ intersects the point $(0,0,1)$.  Let $H=\{y=0, t=1\}$ be the horizontal bottom side of the square $I^2\times \{t=1\}$, and $V=\{x=1, t=1\}$ be the vertical right side of the same square. The curve $\gamma_1$ is represented in $I^2\times \{t=1\}$ by $p+q-1$ disjoint segments.

We will now construct $q$ oriented piecewise linear curves $\lambda_i$ in the boundary of $\partial I^3$, with $i=1,2,\ldots,q$. We proceed iteratively as follows. To construct $\lambda_1$, we first consider a segment that starts at $(1/2,0,0)$ and follows the boundary of $I^3$ along $\{y=0,t=0\}$ up to $(0,0,0)$. At this point, the curve $\lambda_1$ continues along the segment that goes from $(0,0,0)$ along $\{x=0, y=0\}$ up to $(0,0,1)$. The curve continues by following the only segment of $\gamma_1$ that contains $(0,0,1)$. We have to consider several cases. 
\begin{enumerate}
    \item[(1)] If the other endpoint of this segment of $\gamma_1$ intersects the top horizontal side of $I^2\times\{t=1\}$ at a point $(\widetilde x,1 ,1)$ with $\widetilde x\neq 1$, the curve $\lambda_1$ finishes at this point (see for example the left curve of the right picture in Figure \ref{Fig:seg}).
    \item[(2)] If the segment of $\gamma_1$ intersects $V$ along a point $(1, \widetilde y, 1)$ with $\widetilde y\neq 1$, we continue the construction of $\lambda_1$ with the horizontal line segments connecting successively the points $(1,\widetilde y, 1)$, $(1,\widetilde y, 0)$, $(0,\widetilde y,0)$ and $(0,\widetilde y, 1)$ (see for example the left picture of Figure \ref{Fig:seg}). In this case, the construction of the curve $\lambda_1$ is not yet finished.
    \item[(3)] If the segment of $\gamma_1$ intersects $V$ at the point $(1, 1, 1)$, we continue $\lambda_1$ from $(1,1,1)$ through horizontal line segments  from $(1,1,1)$ to $(1,1,0)$ and from there to $(1/2,1,0)$, finishing the construction of $\lambda_1$.
\end{enumerate}  
At this point, if the construction of $\lambda_1$ is not finished we were in case $(2)$. We continue the construction of $\lambda_1$ by repeating the process with the segment of $\gamma_1$ starting at $(0,\widetilde y,1)$. Eventually, the curve $\lambda_1$ reaches the top horizontal side of $I^2\times\{t=1\}$ in cases $(1)$ or $(3)$, ending its construction. 

If there is a point $(\widetilde x, 0, 1)\in H$ in $\gamma_1$ immediately to the right of $(0,0,1)$ along $H$, we construct $\lambda_2$ starting at $(\widetilde x, 0, 1)$ and applying the same recipe as before: we take the line segment of $\gamma_1$ starting at $(\widetilde x, 0, 1)$ and continue according to cases $(1), (2)$ or $(3)$ (see for example the right curve in the right picture of Figure \ref{Fig:seg}).

At the end of this process we have $q$ oriented curves $\lambda_i$, with $\lambda_1$ starting at $(1/2,0,0)$ and $\lambda_q$ ending at $(1/2,1,0)$.

 \begin{figure}[!h]
 \begin{center}
\begin{tikzpicture}
     \node[anchor=south west,inner sep=0] at (0,0) {\includegraphics[scale=0.08]{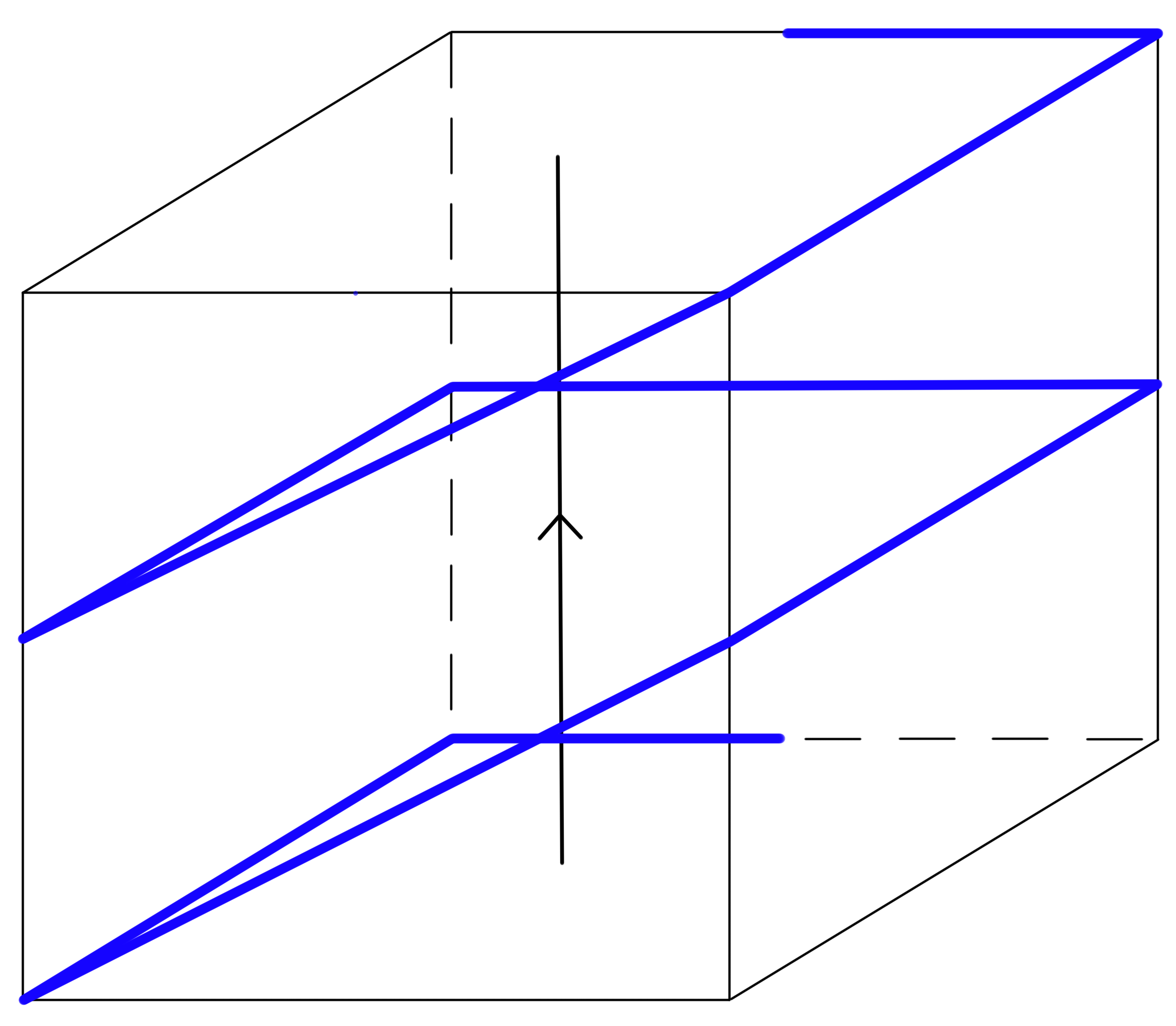}};
        \node[anchor=south west,inner sep=0] at (6.5,0.1) {\includegraphics[scale=0.08]{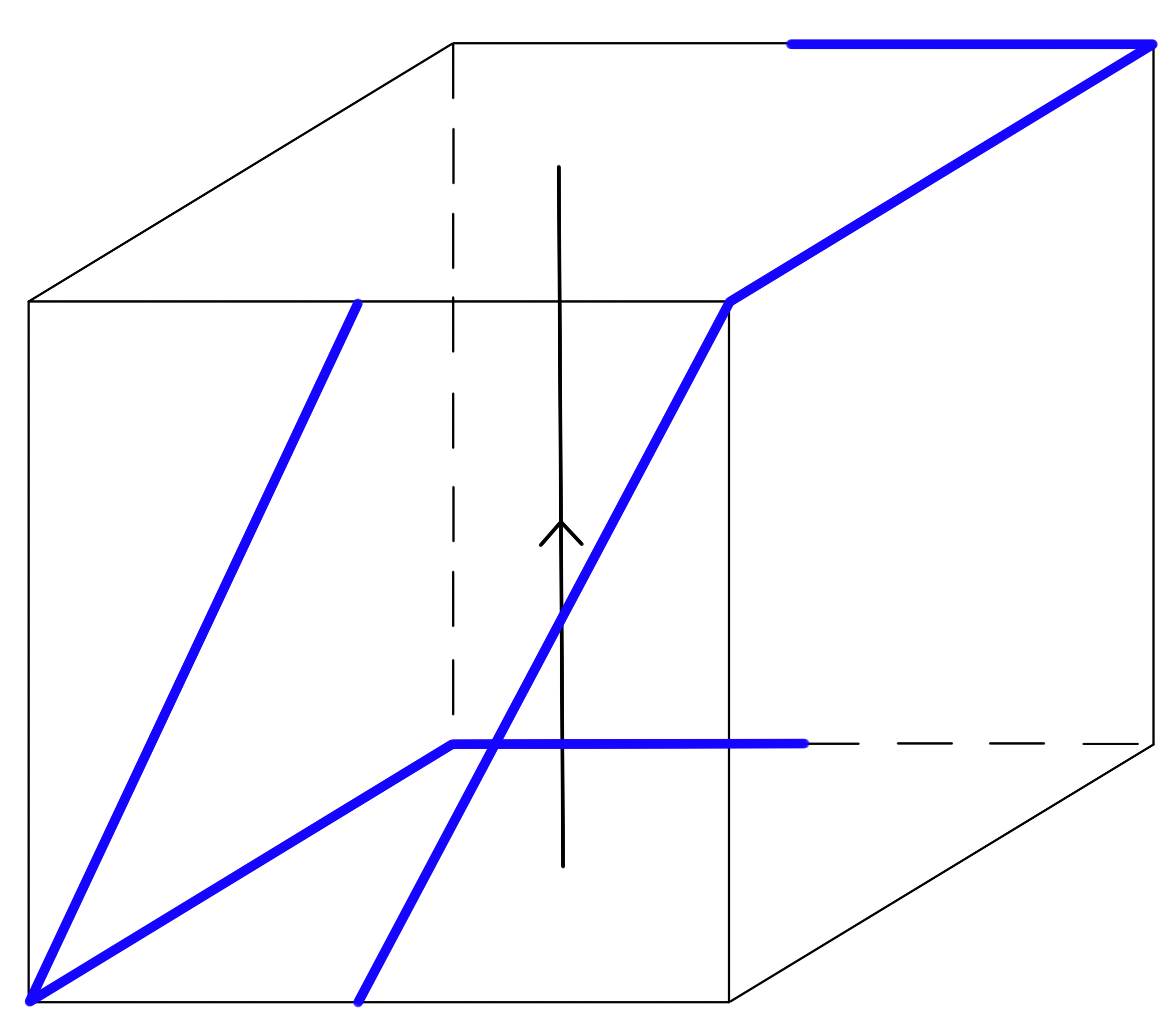}};
    \node at (3.1,3.62) {$V$};
    \draw (3.06,3.07)--(3.06,3.38);
    
    \node at (3.7,0.1) {$H$};
    \draw (3.07,0.12)--(3.4,0.12);
    
    \draw[->] (5.55,1.2)--(5,0.78);
    \node[scale=0.8] at (5.25,0.65) {$t$};

    \draw[->] (5.55,1.2) -- (6.2,1.2);
    \node[scale=0.8] at (6.2, 0.92) {$x$};

    \draw[->] (5.55,1.2) -- (5.55,1.85);
    \node[scale=0.8] at (5.8, 1.85) {$y$};
    
    \node [circle, fill=black, scale=0.35] at (3.25,1.22) {};
    \node[scale=0.7] at (3.65,0.95) {$(1/2,0,0)$};
    
    \node [circle, fill=black, scale=0.35] at (0.12,0.12) {};
    \node[scale=0.7] at (-0.3,-0.15) {$(0,0,1)$};
\end{tikzpicture}
\end{center}
\caption{The curves $\lambda_i$ for $i=1,2,\ldots,q$, in the cases $p=2,q=1$ on the left hand side  and $p=1,q=2$ on the right hand side.
The left hand side corresponds to an example that requires a single curve $\lambda_1$ whose construction goes first through case $(2)$ and then through case $(3)$. The right hand side corresponds to an example where the construction produces two segments: the construction of $\lambda_1$ goes through case $(1)$ and the construction of  $\lambda_2$ starts at $(1/2,0,1)$ and goes through case $(3)$.}
 \label{Fig:seg}
 \end{figure}

Let $(r,\theta)$ be polar coordinates of the square $\{y=0\}$ centered at $\{x=1/2,t=1/2\}$. The  key feature of each path $\lambda_i$ is that it can be parametrized as
$$ \lambda_i (s)= (y(s), r (s), \theta(s)), \enspace s\in [0,1] $$
in a way that $\theta'(s)>0$ for all $s$. Construct for each $\lambda_i$ a surface with boundary $\Sigma_i$ in $I^3$ parametrized by $\phi_i: I^2 \to I^3$ with
\begin{equation}\label{eq:surface}
 \phi_i (s,\rho)=((1-\rho)s+\rho y(s), \rho r(s), \theta(s)).
 \end{equation}
 
  \begin{figure}[!h]
 \captionsetup{justification=centering}
 \begin{center}
\begin{tikzpicture}
     \node[anchor=south west,inner sep=0] at (0,0) {\includegraphics[scale=0.08]{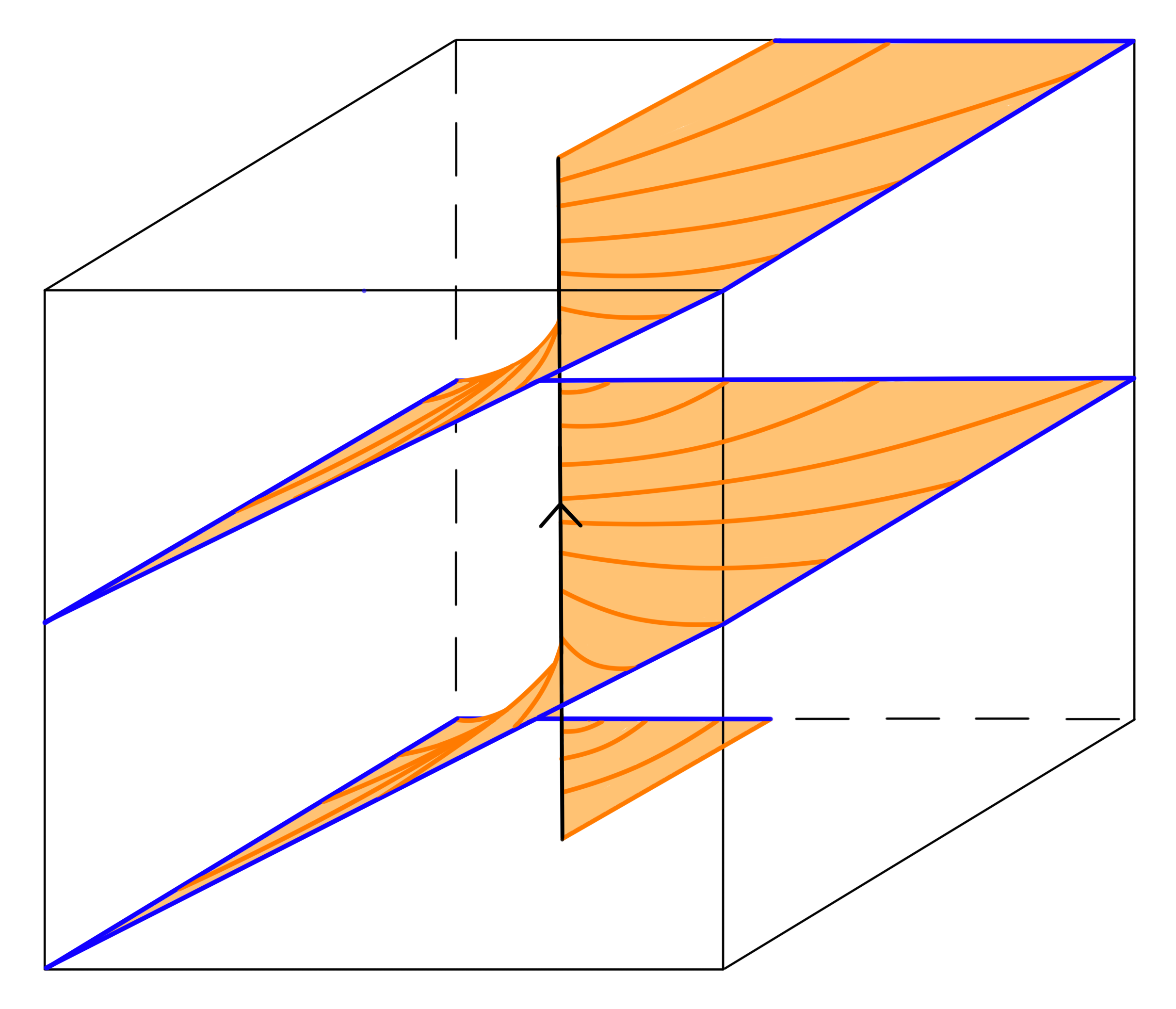}};
        \node[anchor=south west,inner sep=0] at (6.5,0) {\includegraphics[scale=0.08]{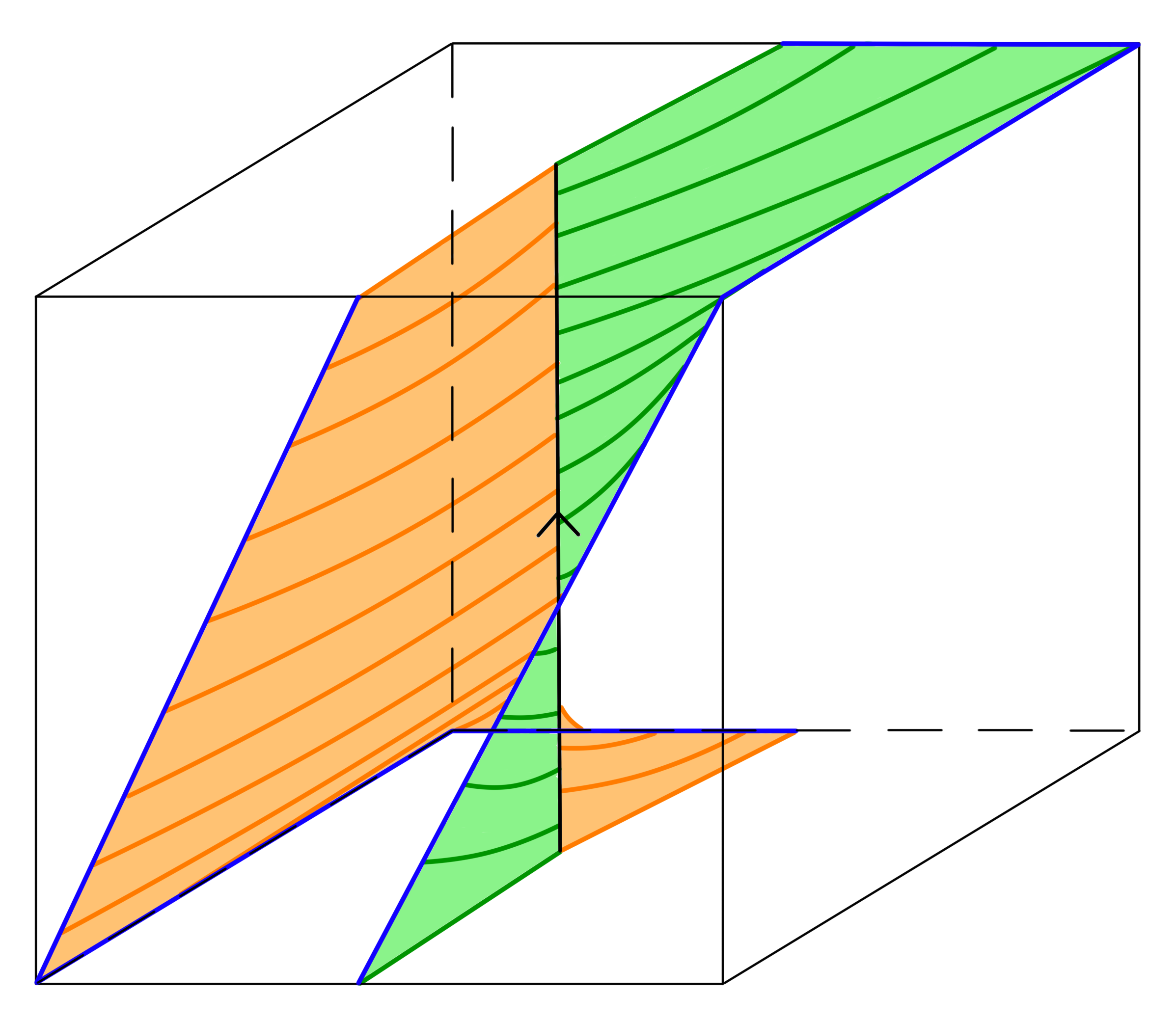}};
\end{tikzpicture}
\centering
\caption{Surfaces for $p=2,q=1$ and $p=1,q=2$}
 \label{Fig:Birk}
\end{center}
 \end{figure}
 
\noindent Choosing well the parameter $s$ of each $\lambda_i$, we can achieve that the surfaces $\Sigma_i$ are pairwise disjoint except along the boundary segment $\{x=1/2,t=1/2\}$.  One option is to choose a parameter $s$ that varies a small quantity when the curve $\lambda_i(s)$ moves along a plane of fixed $y$ coordinate, and $s$ varies approximately as the $y$ coordinate along each segment of $\gamma_1$.

The surfaces $\Sigma_i$ are well defined in the quotient space of $T^2\times I$, and we obtain a continuous surface $\Sigma$ in $T^2\times I$ whose boundary $\partial \Sigma$ is given by $\gamma_1 \subset \{t=1\}$, $p$ curves in $\{t=0\}$ parallel to $\gamma_0$, and the central orbit of the flow given by $\{x=1/2,t=1/2\}$ covered $-q$ times, where we oriented the orbit by the direction of the flow and $\Sigma$ by the unique orientation that makes $X$ positively transverse to its interior. We can easily construct surfaces $\Sigma_i$ that induce a $C^\infty$-immersed surface $\Sigma$ by considering smooth parametrizations of the general form $\widetilde \phi_i(s,\rho)=\left( f_1(s,\rho), f_2(s,\rho), \theta(s) \right)$. It is embedded except maybe along $\{x=1/2,t=1/2\}$ where it is embedded only if $q=1$. The parametrization $\phi_i$ satisfies
\begin{align*}
d\phi_i\left(\pp{}{\rho}\right)&=(y(s)-s)\pp{}{y} + r(s) \pp{}{r} \\
 d\phi_i \left(\pp{}{s}\right) &= (1-\rho+\rho y'(s))\pp{}{y} + \rho r'(s) \pp{}{r} + \theta'(s)\pp{}{\theta},
 \end{align*}
 and since $r(s)$ and $\theta'(s)$ do not vanish in the complement of $\{x=1/2,t=1/2\}$, we deduce that for $\rho \neq 0$, the image of $d\phi$ never contains $\pp{}{y}$.

Note that if $q<0$, the only difference is that we need to adapt the construction with curves $\lambda_i$ that satisfy $\theta'(s)<0$.
\end{proof}

Back to the proof of Proposition~\ref{prop:helix}, we only need to consider the constructed Birkhoff section inside $U=T^2\times [t^*-\varepsilon,t^*+\varepsilon]$ and extend this surface trivially to $T^2\times I$.\\

Let us argue why the section is $\partial$-strong.  Consider the coordinates used in the proof of Lemma~\ref{lem:Birk}. The Poincar\'e map of the periodic orbit $\nu$ is $(x,t)\mapsto (x+f(t), t)$ for some function $f$. We can blow up the periodic orbit $\nu$.  The vector field $X_\nu$ has two periodic orbits $\nu_1$ and $\nu_2$ (that are vertical, parallel to the $y$ direction) that bound two cylinders. Observe that these periodic orbits correspond to the invariant torus that contains $\nu$. In the cylinders bounded by $\nu_1$ and $\nu_2$, the slope of $X_\nu$ at a point of the torus depends on the differential of $f(t)$ in the direction orthogonal to $\nu$ defined by the chosen point. The vector field $X_\nu$ is never horizontal but could be close to horizontal near the center circles of the annuli bounded by $\nu_1$ and $\nu_2$.

Let $S$ be the Birkhoff section obtained by extending the construction in Lemma~\ref{lem:Birk} and $S_\nu$ the induced section after blowing up $\nu$. A priori, there might be points where $S_\nu$ is tangent to $X_\nu$, but $X_\nu$ is never negatively transverse to $S_\nu$. Up to isotopy, we might deform the section near the binding by an isotopy so that these tangencies no longer occur. To do this, we deform the section $S_\nu$ to a new transverse section that is almost vertical near the orbit $\nu_1$ and almost horizontal (at least more horizontal than $X_\nu$ at any point) along some circle transverse to $\partial_t$.

\end{proof}

An example of the last step in the previous explanation is represented in Figure \ref{fig:defsec}, which represents the blow-up at the binding orbit. In this case, the section defines a curve that turns once in the vertical direction and once in the horizontal one, on the left, there are tangency points and on the right there are none.

\begin{figure}[!h]
 \captionsetup{justification=centering}
 \begin{center}
\begin{tikzpicture}
     \node[anchor=south west,inner sep=0] at (0,0) {\includegraphics[scale=0.7]{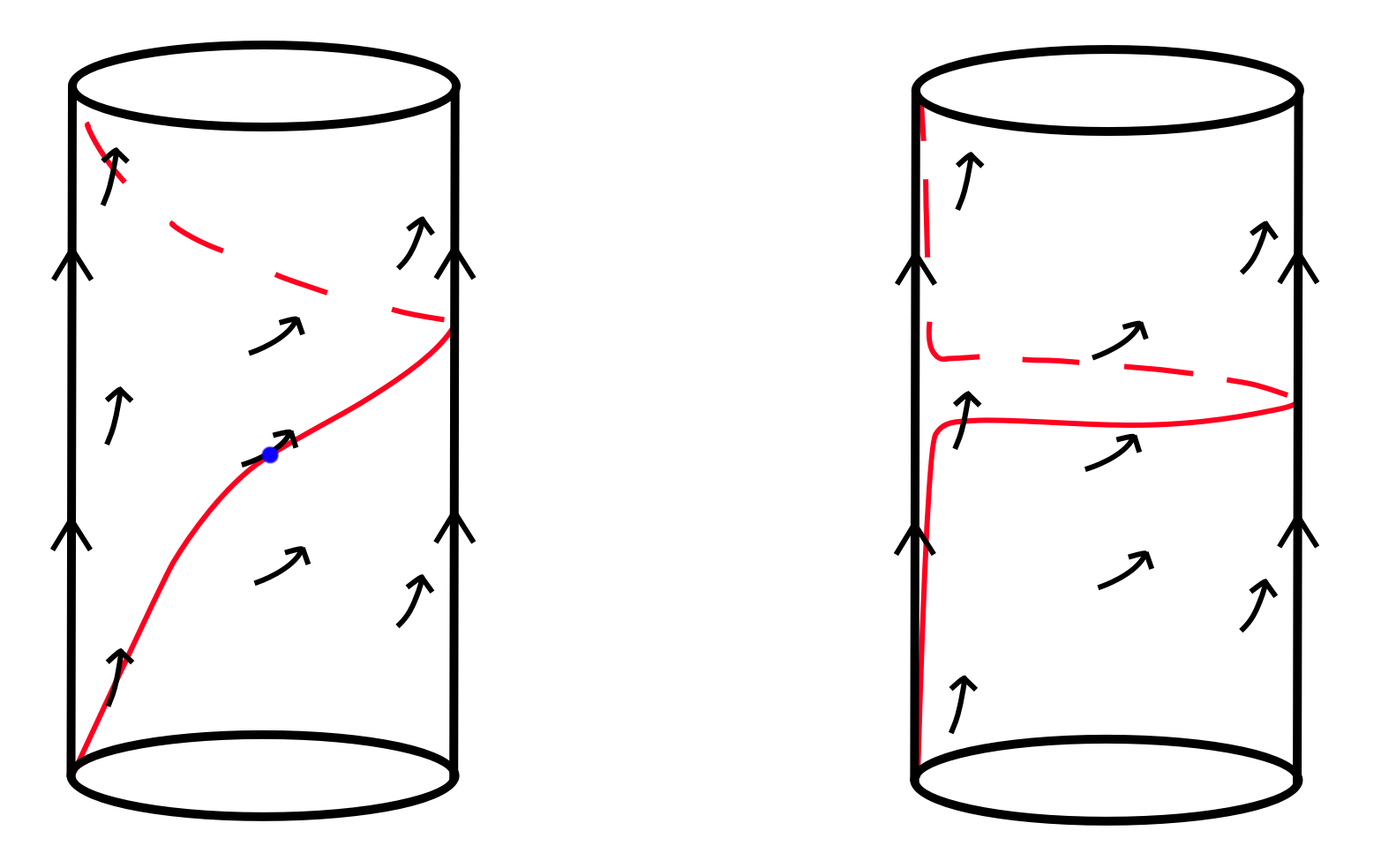}};
          \draw[->] (4.1,3)--(5.1,3);
          \node [color=red] at (2.3,2.5) {$S_\nu$};
          \node at (1.3,1.3) {$X_\nu$};
          \node at (-0.05,3) {$\nu_1$};
           
     \end{tikzpicture}
\centering
\caption{Making a section $\partial$-strong: the tangency points (in blue) are removed by an isotopy of the section.}
 \label{fig:defsec}
\end{center}
 \end{figure}

\begin{Remark}\label{rem:Multcurv}
If $\gamma_1$ is a finite number of parallel copies of a closed curve of primitive homology class, we can take parallel copies of the surface $\Sigma$ constructed in Lemma~\ref{lem:Birk}. This implies that Proposition~\ref{prop:helix} also holds stated as follows: Write $[\gamma_1]=\ell [\gamma_1']$ for a primitive curve $\gamma_1'$, $\ell \in \mathbb{Z}$, and $[\gamma_1']=p[\nu] + q [\gamma_0]$. There is a Birkhoff section $\Sigma$ such that $\Sigma\cap \{t=0\}$ is given by $\ell p$ parallel copies of $\gamma_0$, $\Sigma\cap \{t=1\}$ is given by $\gamma_1$, and the binding of $\Sigma\cap \{t=1\}$ is $\nu$ covered $-\ell q$ times.
\end{Remark}

\begin{Remark}\label{rem:embDeh}
Similar surfaces were first considered in \cite{VHM} in the context of open book decompositions for contact structures. They were later used in dynamical systems by Dehornoy \cite{Deh, Deh2} to classify global surfaces of section of the geodesic flow on the flat torus. His construction uses several binding orbits instead of one, and yields surfaces that are embedded up to the boundary (instead of surfaces that are only immersed along the boundary, as in our construction). Those are constructed by piling up vertically and horizontally copies of the surface obtained in Lemma~\ref{lem:Birk} when choosing a linear curve $\gamma$ with $p=1,q=1$. Using this approach, one obtains an embedded Birkhoff section, at the cost of needing several parallel binding components along the torus $t=\{1/2\}$.

 In \cite{Deh, Deh2}, the name {\it helix box} refers to the construction in Lemma~\ref{lem:Birk} for a curve $\gamma$ such that $p=1,q=1$. Since it is a natural generalization, we use the same name for the construction in Lemma~\ref{lem:Birk}, independently of the curve $\gamma$. We proceed to the proof of the main result of this section, obtained by concatenating suitable helix boxes.
 \end{Remark}

\begin{prop}\label{prop:BirkT2}
Let $X$ be a non-singular vector field in $T^2\times I$ that is $T^2$-invariant and whose slope is non-constant. Let $\Gamma_0,\Gamma_1$ be two families of oriented embedded closed curves such that $\Gamma_0 \subset T^2\times \{0\}$, $\Gamma_1\subset T^2\times \{1\}$, and such that $X|_{t=0}$ and $X|_{t=1}$ is respectively positively transverse to $\Gamma_0$ and $\Gamma_1$. Then there exists a $\partial$-strong Birkhoff section $\Sigma$ of $X$ such that $\Sigma \cap \{t=1\}=\Gamma_1$ and $\Sigma \cap \{t=0\}=\Gamma_0$.
\end{prop}

\begin{proof}
We might assume that $\Gamma_0$ and $\Gamma_1$ are just parallel copies of linear closed curves $\gamma_0, \gamma_1$ that have primitive integer homology classes in $H_1(T^2;\mathbb{Z})$, respectively. In homology, we have $[\Gamma_0]=r[\gamma_0]$ and $[\Gamma_1]=\ell[\gamma_1]$. As before, we can consider these curves to be linear constant slope curves in a torus.

Let $k(t)$ denote the slope of the vector field $X$ on $T^2\times \{t\}$, with respect to a coordinate system $(x,y,t)$ in $T^2\times I$. Let $\varepsilon>0$ be such that the slope $k(t)$ is not constant in $[0,\varepsilon]$ and in $[\varepsilon, 1]$, varies only by a small value and satisfies $k'(t)\geq 0$ for $0\leq t\leq \varepsilon$ (if $k'(t)\leq 0$, an analogous argument works). In particular, we can assume that $\Gamma_0\times [0, \varepsilon]$ is transverse to $X$. Choose two primitive integer homology classes $b_1, b_2$ of $T^2$ represented by two positively oriented periodic orbits of the flow on two tori $\{t=s_1\}$ and $\{t=s_2\}$ where $X$ has rational slope and such that $0<s_1<s_2<\varepsilon$. We might further impose that $b_1,b_2$ are generators of $H_1(T^2;\mathbb{Z})$ and that $[\gamma_0],b_1$ are also generators of $H_1(T^2;\mathbb{Z})$. We can choose the coordinates $x,y$ in $T^2$ such that $\pp{}{x}$ is parallel to $\gamma_0$ and $\pp{}{y}$ is parallel to $b_1$. Since the slope of the Reeb vector field is of positive derivative for $t\in [0,\varepsilon]$, it follows that $b_2$ can be chosen such that $b_2=Nb_1+[\gamma_0]$ for some $N>>0$.

Fix a small number $\theta_0>0$ and set $t_0=0$. Partition the interval $[\varepsilon,1]$ into (possibly just one) subintervals 
$$[t_1,t_2], [t_2,t_3],..., [t_{n-1},t_n],$$
where $t_1=\varepsilon$ and $t_n=1$, so that in each interval the slope is non-constant and varies only by some amount smaller in absolute value than $\theta_0$.  We can take $\theta_0$ such that $\Gamma_1\times I$ is transverse to $X$  on $T^2\times [t_{n-1},1]$.  Construct a family of closed curves $\sigma_i$ with $i=1,...,n-1$ in $T^2\times \{t_i\}$ such that:
\begin{itemize}
\item[-] $[\sigma_1]=p'b_1+q'[\gamma_0]$ for some coprime integers such that $p'<0$ and $q'\geq r$, and $[\sigma_1],b_2$ are generators of $H_1(T^2;\mathbb{Z})$,
\item[-] $\sigma_i\times I$ is transverse to $X$ for $t\in[t_{i-1},t_{i+1}]$ for each $i=1,...,n-1$.
\end{itemize}
This is possible because the slope of the vector field does not vary more than $\theta_0$ in each interval, iterating choose $\sigma_i$ for $i=2,...,n-1$ to be a curve whose slope is approximately minus the inverse of the slope of the vector field in $T^2\times \{t_i\}$ (with respect to the coordinates $(x,y)$ in $T^2 \times I$). For $i=1$ there are many choices for $\sigma_1$. 

Having the curves $\sigma_i$ we choose, in each domain $T^2\times (t_i, t_{i+1})$, a periodic orbit of the flow whose homology class together with $[\sigma_i]$ gives a base of $H_1(T^2;\mathbb{Z})$. These will be the bindings of the helix boxes that we will construct in each $T^2\times [t_i, t_{i+1}]$.

We apply Proposition~\ref{prop:helix} first in $T^2\times [t_{n-1},t_n=1]$ where the curve in $T^2\times \{1\}$ is $\Gamma_1$ (by Remark~\ref{rem:Multcurv}, Proposition~\ref{prop:helix} applies even if $\Gamma_1$ is not connected). We find a $\partial$-strong Birkhoff section whose boundary along $t=1$ is $\Gamma_1$ and along $t=t_{n-1}$ is a finite collection of parallel copies of $\sigma_{n-1}$. Iteratively, we apply Proposition~\ref{prop:helix} to each domain $T^2\times [t_i,t_{i+1}]$ with $i=n-2,...,1$, choosing the boundary of the Birkhoff section constructed in $[t_{i+1},t_{i+2}]$ as the curve in $T^2\times \{t_{i+1}\}$ (it is given by a finite number of copies of $\sigma_{i+1}$), and $\sigma_i$ as the curve in $T^2\times \{t_i\}$. Again by Remark~\ref{rem:Multcurv}, there is no issue in considering multiple copies of a curve. After reaching $t=t_1=\varepsilon$, we constructed a $\partial$-strong Birkhoff section $\Sigma$ in $T^2\times [\varepsilon,1]$ that intersects $\{t=1\}$ along $\Gamma_1$, and intersects $\{t=\varepsilon\}$ along some finite number of parallel copies of $\sigma_1$. To simplify the notation, denote $\gamma_\varepsilon$ the curve $\sigma_1$, and $\Gamma_\varepsilon$ the finite collection of parallel copies of $\gamma_\varepsilon$ given by the intersection of the Birkhoff section with $\{t=\varepsilon\}$.  

By construction $[\gamma_\varepsilon]= p' b_1 + q' [\gamma_0]$,
for some coprime integers $p',q'$ satisfying $p'<0$ and $q'\geq r$. Since
$[\Gamma_\varepsilon]= s[\gamma_\varepsilon]=pb_1+q[\gamma_0]$ 
for some positive integer $s$, we have $p<0$ and $q\geq r$.
We can decompose the homology class of $\Gamma_\varepsilon$ as
 $$[\Gamma_\varepsilon]= [\Gamma_0] + (p-(q-r)N)b_1 + (q-r)b_2,$$
 using that $b_1,b_2$ are generators. In vectorial notation, we have $[\gamma_0]=(0,1)$, $b_1=(1,0)$, and the equality above is tantamount to
 $$ (p,q)= (0,r) + (p-(q-r)N)(1,0) + (q-r) (N,1). $$
 Choose some $\varepsilon'\in (s_1,s_2)$ and a set of closed (linear) curves $\Gamma_\varepsilon'$ in $T^2\times \{\varepsilon'\}$ such that $[\Gamma_{\varepsilon'}]=[\Gamma_0]+(p-(q-r)N)b_1$. We would like to apply Proposition~\ref{prop:helix} in $T^2\times [0,\varepsilon']$ and in $T^2\times [\varepsilon',\varepsilon]$, choosing as curves $\Gamma_0, \Gamma_{\varepsilon'}$ and with binding in $T^2\times \{s_1\}$, and $\Gamma_{\varepsilon'}, \Gamma_{\varepsilon}$ with binding in $T^2\times \{s_2\}$. The curves and bindings are schematically depicted in Figure~\ref{Fig:LongHelix}.
 
   \begin{figure}[!h]
 \captionsetup{justification=centering}
 \begin{center}
\begin{tikzpicture}
     \node[anchor=south west,inner sep=0] at (0,0) {\includegraphics[scale=0.08]{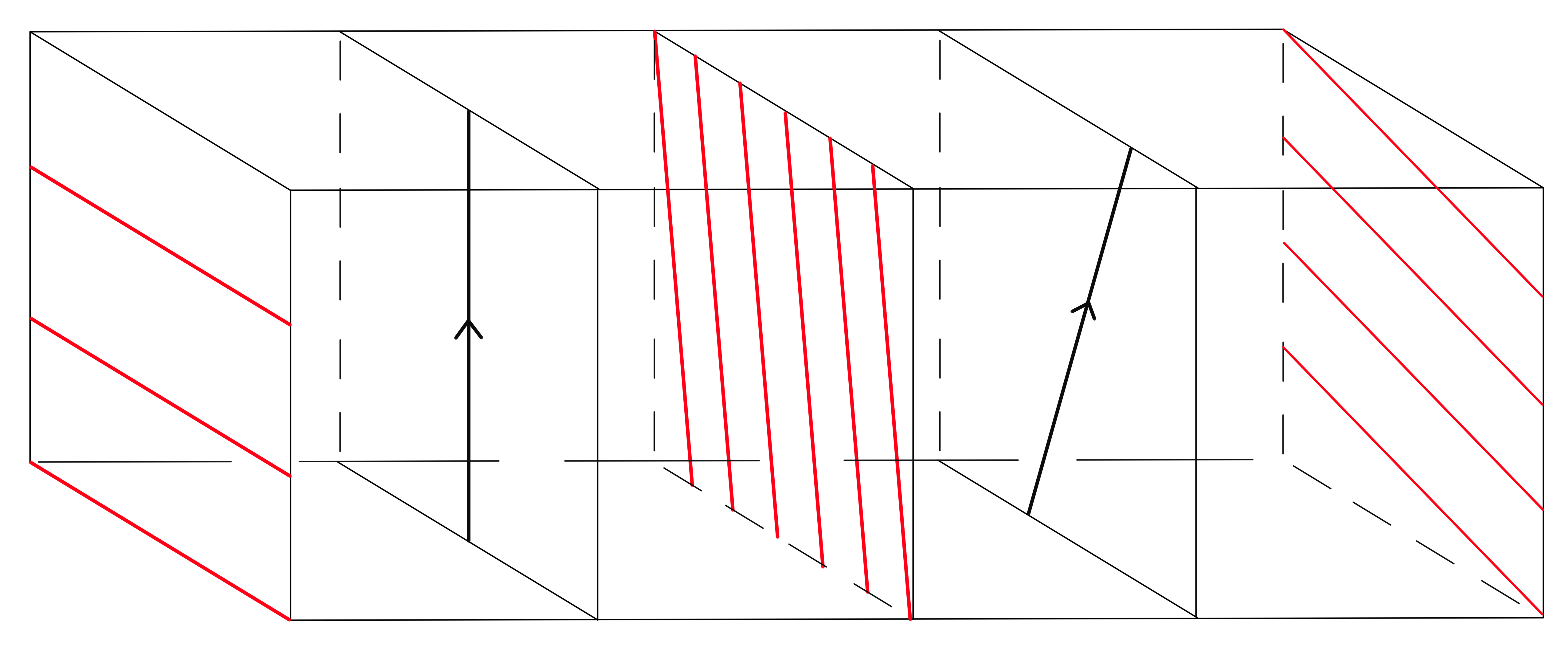}};
     
     \node[scale=0.8] at (1.8,-0.1) {$t=0$};
     \node[scale=0.8] at (4.3,-0.1) {$t=s_1$};
     \node[scale=0.8] at (6.5,-0.1) {$t=\varepsilon'$};  
     \node[scale=0.8] at (8.5,-0.1) {$t=s_2$};
     \node[scale=0.8] at (10.7,-0.1) {$t=\varepsilon$};     
     
     \node [color=red] at (-0.2,3.5) {$\Gamma_0$};
     \node [color=red] at (4.6,4.7) {$\Gamma_{\varepsilon'}$};     
     \node [color=red] at (9,4.7) {$\Gamma_{\varepsilon}$};
     
     \end{tikzpicture}
\centering
\caption{Last step in $T^2\times [0,\varepsilon]$}
 \label{Fig:LongHelix}
\end{center}
 \end{figure}
 
 To be able to do it, we need to check that $\Gamma_{\varepsilon'}\times I$ is transverse to the flow for $t\in [s_1,s_2]$.
 
 We claim that the condition $p<0, q>r$ ensures that $\Gamma_{\varepsilon'}\times I$ is transverse to the vector field for $t\in [s_1, \varepsilon]$. Recall that $s_1$ is such that in $T^2\times \{s_1\}$ the orbits of $X$ have homology $b_1$. The transversality follows from the fact that the vector field ``rotates" in the direction that increases the angle between $X$ and $\Gamma_{\varepsilon'}\times [s_1,\varepsilon]$ as $t$ increases. This is represented in Figure~\ref{Fig:LongHelix}: the vector field rotates clockwise as $t$ increases, becoming increasingly transverse to $\Gamma_{\varepsilon'} \times [s_1,\varepsilon]$. 
 
 Formally, the vector field along any torus $T^2 \times \{t^*\}$ with $t^*\in [s_1,\varepsilon]$ is of the form $X_{t^*}=m\pp{}{x} + n\pp{}{y}$ where $n,m$ are real numbers such that $n>>m>0$. On the other hand, the section is given by an integral curve of the vector field $Y=r\pp{}{x}+\big( p-(q-r)N\big) \pp{}{y}$. The determinant of the matrix whose columns are the coefficients of these vector fields is
$$\det \begin{pmatrix} 
r & m \\
 p-(q-r)N & n
\end{pmatrix}= rn-m\big( p-(q-r)N\big).$$
It follows from the fact that $p<0$ and that $q\geq r$ that this determinant is always positive, as claimed. 

We are now able to apply Proposition~\ref{prop:helix} in $T^2\times [0,\varepsilon']$, and construct a $\partial$-strong Birkhoff section of the flow such that along $t=0$ it defines the curves $\Gamma_0$ and at $t=\varepsilon'$ it defines $\Gamma_{\varepsilon'}$, and whose binding is a periodic orbit of the flow along $t=s_1$. The key fact in this step is that the homology of $\Gamma_{\varepsilon'}$ expressed in the base $\gamma_0, b_1$ has a coefficient $r$ in $\gamma_0$. This ensures that along $t=0$ the section can be chosen to coincide with $\Gamma_0$, i.e. with exactly $r$ copies of $\gamma_0$.

Finally, we have
 $$ [\Gamma_\varepsilon]= [\Gamma_{\varepsilon'}]+ (q-r)b_2,$$
 and the key fact is again that in the base $[\gamma_{\varepsilon'}], b_2$, both sets of curves have the same coefficient in $[\gamma_{\varepsilon'}]$. We apply again Proposition~\ref{prop:helix} to construct a $\partial$ Birkhoff section of the vector field in $T^2\times [\varepsilon',\varepsilon]$ with a binding component which is a closed curve of the flow in $\{t=s_2\}$. This Birkhoff section coincides with $\Gamma_{\varepsilon'}$ and $\Gamma_{\varepsilon}$ when intersected with $t=\varepsilon'$ and $t=\varepsilon$ respectively. 
 
 In the construction above we pasted together different local Birkhoff sections identifying parts of the boundaries that are far away from the binding orbits. Thus the construction gives a Birkhoff section that is piecewise smooth and smooth near the binding. Since being positively transverse to a vector field is an open condition, we can smooth the Birkhoff section finishing the proof Proposition~\ref{prop:BirkT2}.
\end{proof}

\begin{Remark}
Proposition~\ref{prop:BirkT2} holds as well if we want a global surface of section instead of a Birkhoff section (i.e. the surface is embedded even along the boundary). This can be done by using several bindings, see Remark~\ref{rem:embDeh}.
\end{Remark}

\subsection{Construction of broken books and Birkhoff sections} \label{sec:bbook}

Let us first recall the definition of broken book decomposition, introduced in \cite{CDR}.

\begin{defi}\label{defn:BBD}
A degenerate broken book decomposition of a closed $3$-manifold $M$ is a pair $(K,\mathcal{F})$ such that:
\begin{itemize}
\item[-] $K$ is a link, called the binding, 
\item[-] $\mathcal{F}$ is a smooth cooriented foliation of $M\setminus K$ such that each leaf $S$ of $\mathcal{F}$ is properly embedded in $M\setminus K$ and admits a compactification $\overline{S}$ in $M$ that we call a page. The boundary of $S$ is contained in $K$.
\item[-] $K=K_r\sqcup K_b$, respectively the radial and broken parts of the binding. Each $k\in K_r$ has a tubular neighborhood $\mathcal{U}$ in which $\mathcal{F}|_{\mathcal U}$ is a radial foliation: every leaf of $\mathcal{F}|_{\mathcal U}$ is an annulus with one of its  boundary components covering $k$ a certain number of times. Each $k\in K_b$, the broken part of the binding, has a tubular neighborhood $\mathcal{V}$ such that every leaf of $\mathcal{F}|_{\mathcal{V}}$ is an annulus, some are radial meaning that one of the boundary components covers $k$ and some of them have both boundary components in $\partial \mathcal{V}$. 
\end{itemize}
\end{defi}

The link $K$ in the above definition is not required to be oriented. Consider a connected component $k\in K$. Each leaf of $\mathcal{F}$ having a boundary component on $k$ induces an orientation on $k$ and if $k\in K_r$ different leaves induce different orientations on $k$ (see Remark~\ref{rem:orientation}).

A rational open book decomposition of a closed 3-manifold is given by a link $K$ and a fibration of $M\setminus K$ over $\mathbb{S}^1$. Notice that if the broken book decomposition has no broken components, it is a rational open book decomposition. In this case, $M\setminus K$ fibers over $\mS^1$ and the pages of the foliation $\mathcal{F}$ coincide with the fibers.

From Definition~\ref{defn:BBD}, we have that if $\mathcal{V}$ is atubular neighborhood  of a broken component  $k\in K_b$ the foliation $\mathcal{F}|_{\mathcal{V}}$ has two types of leaves: radial ones that are annuli with a boundary component in $k$ and the other boundary component in $\partial {\mathcal{V}}$; and {\it hyperbolic} ones that are annuli with both boundary components in $\partial {\mathcal{V}}$. 
Along a broken component $k$ of the binding, consider a disc transverse to $k$ and to $\mathcal{F}$ contained in the tubular neighborhood $\mathcal V$. We  call radial/hyperbolic sectors to the connected components of such a disc that belong to radial/hyperbolic leaves. Broken book decompositions, without the degenerate adjective, have an additional finiteness hypothesis on the sectors at each broken component of the binding.  

\begin{defi}\label{defn:BBD2}
    A broken book decomposition is a degenerate broken book decomposition such that  every $k\in K_b$ has exactly four hyperbolic sectors.
\end{defi}

\begin{Remark}
Radial sectors are closed. A radial sector might contain just one leaf of $\mathcal{F}|_{\mathcal{V}}$. Moreover, the definitions of radial and hyperbolic leaves are only local. 
\end{Remark}

There is a finite number of pages that do not belong to the interior of one-parameter families of homeomorphic pages, which are called {\it rigid pages}. The boundary of a rigid page must contain  broken binding components. Each connected component of the complement of the rigid pages fibers over $\mathbb{R}$ and the fibers can be taken to be the leaves of $\mathcal{F}$. 

\begin{defi}
    Given a vector field $X$ on a 3-manifold equipped with a broken book decomposition $(K,\mathcal{F})$, we say that the broken book carries (or supports) the vector field $X$ if the binding $K$ is composed of periodic orbits of $X$, while the other orbits are positively transverse to the leaves of $\mathcal{F}$.
\end{defi}

If a vector field $X$ admits a Birkhoff section $S$, then it is carried by an open book decomposition whose binding is $\partial S$ and whose pages are diffeormorphic to $S$.

\begin{Remark}
In Theorem~\ref{thm:bbook}, a periodic orbit of a vector field $X$ carried by a broken book decomposition that belongs to $K_b$ will always be non-degenerate and hyperbolic. In this case, it can be either positive or negative hyperbolic according to the sign of the eigenvalues of the linearised Poincar\'e map.

In the construction below, $K_b$ will be contained in the contact part of the Reeb vector field of a SHS, and hence the conclusion above follow from the contact non-degenerate hypothesis.
\end{Remark} 

\begin{Remark}\label{rem:orientation}
    If a periodic orbit $k$ of $X$ belongs to $K_b$, two radial leaves in adjacent radial sectors induce opposite orientations on $k$.
\end{Remark}
We cite two definitions from \cite{CDR} for a smooth non-singular vector field $X$ on a closed 3-manifold $M$. Recall that given a periodic orbit $\gamma$ of $X$, we denote by $\Sigma_\gamma$ the unit normal bundle~$(TM_\gamma/T\gamma)/\R_+$ to~$\gamma$ and by~$M_\gamma$ the normal blow-up of~$M$ along~$\gamma$.

\begin{defi}
Let $S$ be a (not necessarily connected) 
transverse surface with boundary.
\begin{itemize}
\item[-] An orbit~$\gamma$ of~$X$ is {\it asymptotically linking}~$S$ if for every $T\in\R$ the arcs~$\gamma([T, +\infty))$ and~$\gamma((-\infty, T])$ intersect $S$. 

\item[-] If $\gamma$ is a non-degenerate periodic orbit in $\partial S$, consider its unit normal bundle~$\Sigma_\gamma$. The {\it self-linking} of $\gamma$ with~$S$ is the rotation number of the extension of~$X$ to~$\Sigma_\gamma$, with respect to the 0-slope given by the closed curve $\partial_\gamma S$ in $\Sigma_\gamma$. 
\end{itemize}

\end{defi}

As for Birkhoff sections, we have a notion of $\partial$-strongness for broken books.

\begin{defi}\label{defn-bbddelta}
Consider a broken book decomposition $(K,\mathcal{F})$ of a closed 3-manifold $M$. A non-singular vector field $X$ on $M$ is $\partial$-strong carried by $(K,\mathcal{F})$ if it is carried by  $(K,\mathcal{F})$ and each leaf of $\mathcal{F}$ is a $\partial$-strong transverse surface.
\end{defi}

We proceed with the proof of the first part of Theorem~\ref{thm:mainBirk}.

\begin{theorem}\label{thm:bbook}
Let $(\lambda,\omega)$ be a SHS in $\mathcal{B}$ on a closed 3-manifold $M$ whose Reeb vector field is $X$. 
 Then there is a $C^2$-neighborhood of $(\lambda,\omega)$ in the set of SHS, such that every Reeb vector field of a SHS in this neighborhood is supported by a $\partial$-strong broken book decomposition. 
\end{theorem}


\begin{proof}
Let $U, N_0, N_c$ be a structural decomposition for which $(\lambda,\omega)$ is contact non-degenerate. If $U=\emptyset$, the result is immediate. We hence assume that $U\neq \emptyset$. Let us recall the ingredients of the construction of a broken book decomposition supporting a vector field $X$. These correspond to the properties enumerated in \cite[Lemma 3.6]{CDR} and are used in the construction of a broken book decomposition. We refer to \cite[Section 3]{CDR} for the proof.

\begin{prop}\label{prop:ingredients}
Let $X$ be a non-singular vector field on a closed 3-manifold $M$. Assume there is a finite collection of $\partial$-strong transverse surfaces with boundary $S_1, S_2, \ldots S_\ell$ such that:
\begin{enumerate}
    \item the {interior of the} transverse surfaces {are pairwise disjoint};
    \item every orbit intersects $\cup_i S_i$;
    \item $\cup_i\partial S_i= K$;
    \item $M\setminus \left(\cup_i S_i\right)$ fibers over $\mathbb{R}$ and the flow is transverse to the fibers;
    \item if an orbit of $X$ is not asymptotically linking $\cup_i S_i$, it converges to one of their boundary components which is a non-degenerate hyperbolic periodic orbit $\gamma$ with tubular neighborhood $\mathcal{V}$. In this case, each one of the quadrants transversally delimited by the stable and unstable manifolds of $\gamma$ is intersected by at least one $S_i$ such that a connected component of $S_i\cap \mathcal{V}$ contained in this quadrant has  $\gamma$ as a boundary component. 
\end{enumerate}
Then there is a $\partial$-strong broken book decomposition $(K,\mathcal{F})$ supporting $X$ and having the surfaces $S_i$ as pages.
\end{prop}

The binding non-degenerate hyperbolic orbits in Proposition~\ref{prop:ingredients} (5) are the orbits in $K_b$, that is the broken components of the binding of the broken book obtained. Given a vector field carried by a broken book decomposition, one can choose a finite collection of pages satisfying Proposition~\ref{prop:ingredients}: the set of rigid pages is one possibility. 
 
Below, we construct piecewise transverse surfaces to the vector field $X$ by constructing them in the regions $N_c$, $U$ and $N_0$. These local constructions are pasted together and can be smoothed as in the proof of Proposition~\ref{prop:BirkT2}.

Let $W$ denote a connected component of $N_c$. Near $\partial W$, the Reeb vector field is parallel to a linear vector field of constant irrational slope: arguing as in the proof of Lemma~\ref{lem:reebpseudo} we can compactify $W$ into a closed 3-manifold $\overline W$, and find a contact form $\alpha$ in $\overline W$. The invariant boundary components of $W$ yield closed non-degenerate elliptic orbits of the Reeb vector field $R_\alpha$ that are surrounded by a foliation of invariant tori $T_\rho$ for a real parameter $\rho$. Let $\Gamma$ be the collection of these periodic orbits of $R_\alpha$.
By construction, the vector field $X$ restricted to the interior of $N_c$ coincides with $R_\alpha$ in $\overline W \setminus \Gamma$.

Applying Theorem 1.1 in \cite{CDR}, we know that $R_\alpha$ is carried by a $\partial$-strong broken book decomposition $(K,\mathcal{F})$. Any orbit $\gamma \in \Gamma$ is an elliptic periodic orbit and hence it cannot be a broken component of $K$. It follows that either $\gamma$ is a radial component of $K$ or $\gamma$ is transverse to $\mathcal{F}$. In both cases, $\mathcal{F}$ induces in each invariant torus $T_\rho$ a foliation by closed curves transverse to $R_\alpha$, which is a linear irrational vector field in each torus. We deduce that in each connected component of $N_c$ the Reeb vector field $X$ is carried by a broken book decomposition, whose pages are transverse to the boundary tori and such that $K$ is in the interior of $N_c$. 

Let us still denote by $(K,\mathcal{F})$ the $\partial$-strong broken book decomposition of $N_c$: that is $(K,\mathcal{F})$ restricts to a $\partial$-strong broken book decomposition in each component of $N_c$.

On the other hand, in each connected component of $N_0$, the Reeb vector field is transverse to the fibers of a surface bundle over the circle. This means that in the boundary of any connected component of $V={U\setminus \operatorname{int}(N)}$, which connects two boundary components of $N$, there are two different induced homology classes of curves: the intersection of any page of $(K,\mathcal{F})$ with the boundary torus of $V$ (if that boundary torus belongs to $N_c$), or the intersection of any fiber of the surface bundle with the torus (if that boundary torus belongs to $N_0$).  

Choose the finite collection of rigid pages of $(K,\mathcal{F})$ in $N_c$ and denote them by $P=P_1\cup \cdots \cup P_k$. Observe that since $N_c$ has a boundary, the surfaces $P_i$ might have boundary components in $\partial N_c$. 

Take a connected component $V_i\cong T^2\times I$ of $V$ and set $\partial V_i=T_0\cup T_1$ where $T_k\cong T^2\times \{k\}$ for $k=0,1$. If $T_k \subset \partial N_c$ let $\Gamma_k$ be the collection of closed curves $P\cap T_k$. If $T_k\subset \partial N_0$ let $\Gamma_k$ be the intersection of a fiber with $T_k$. We obtain two families of closed curves $\Gamma_0$ and $\Gamma_1$ transverse to the Reeb vector field along $T_0$ and $T_1$ respectively. Applying Proposition~\ref{prop:BirkT2}, we construct a $\partial$-strong Birkhoff section in $V_i$ that glues together the surfaces in $N$ along $V_i$.

We can do this for each $V_i$, obtaining a finite collection $S=S_1\cup\cdots \cup S_\ell$ of $\partial$-strong transverse surfaces with boundary.  We claim that $S$ satisfies Proposition~\ref{prop:ingredients}. By construction, the surfaces have disjoint interiors and intersect all the orbits of $X$. Observe that in the closure of $M\setminus N_c$ the restriction of  $S$ is a Birkhoff section, hence the orbits that are not asymptotically linking $S$ are contained in $N_c$ and by the choice of $P$ they satisfy (5) of Proposition~\ref{prop:ingredients}. The same argument implies that the complement of $S$ fibers over $\mathbb{R}$.

The existence of a $\partial$-strong broken book decomposition carrying $X$ follows from Proposition~\ref{prop:ingredients}. 

\smallskip

It remains to show that there exists an open $C^2$-neighborhood (in the set of SHS) of $(\lambda,\omega)$, where each Reeb vector field is also carried by a $\partial$-strong broken book decomposition. Let $(\widetilde \lambda, \widetilde \omega)$ be a SHS sufficiently $C^2$-close to a SHS $(\lambda,\omega)\in \mathcal{B}$ that is contact non-degenerate for some structural decomposition $U, N_0, N_c$. If $U=\emptyset$, then either the Reeb vector field of $\widetilde X$ of $(\widetilde \lambda, \widetilde \omega)$ admits a global section (as this is a $C^0$-open property for a vector field), or $\widetilde \lambda$ is a contact form (as this is an open property as well) that is $C^1$-close to a non-degenerate one, and then it admits a $\partial$-strong broken book decomposition by \cite[Theorem 1.1]{CDR}. If $U\neq \emptyset$, by \cite[Theorem 3.7]{CV1}, the SHS $(\widetilde \lambda, \widetilde \omega)$ will have $T^2$-invariant integrable regions $K_i$ inside each component $V_i$ of $V=U\setminus \text{int}(N)$ of almost full measure in $V$. Since the slope of the Reeb vector field of $(\lambda,\omega)$ on $V$ was non-constant, the slope of the Reeb vector field of $(\widetilde \lambda, \widetilde \omega)$ has non-constant slope if the SHS is $C^2$-close to $(\lambda,\omega)$. The complement of $K=\sqcup K_i$ is $C^1$ close and diffeomorphic to $N_0\sqcup N_c$. In $N_0$, the Reeb vector field still admits a global section, since it is $C^1$-close to a suspension flow. In the contact region $N_c$, the Reeb vector field is $C^1$-close to the Reeb vector field of $(\lambda, \omega)$, and the contact form $\widetilde \lambda$ is $C^2$-close to $\lambda$. The broken book decomposition that carries the Reeb vector field of $(\lambda,\omega)$ is such that in the contact region, the binding components are non-degenerate periodic orbits and the flow is $\partial$-strong carried by the broken book. It was shown in \cite[Section 4.7]{CDR} that the Reeb vector field defined by a contact form that is $C^2$-close to a non-degenerate contact form is also supported by a $\partial$-strong broken book decomposition. We deduce that even if 
$(\widetilde \lambda, \widetilde \omega)$ is a priori not contact non-degenerate, the Reeb vector field is supported by a $\partial$-strong broken book decomposition in the contact region. 
The slope in the integrable region is non-constant in each connected component, so we can still construct a $\partial$-strong broken book decomposition supporting the Reeb vector field of $(\widetilde \lambda, \widetilde \omega)$ arguing step by step as in the contact non-degenerate case.
\end{proof}

\begin{Remark} \label{rem:bbook}
We only required the contact non-degeneracy to deduce that the Reeb vector field in the contact regions is carried by a $\partial$-strong broken book decomposition. Additionally, by construction, if the broken book decomposition $(K,\mathcal{F})$ given in the contact region $N_c$ is a rational open book decomposition, the previous theorem implies that the Reeb vector field of $(\lambda,\omega)$ is carried by a rational open book decomposition. Indeed, the only broken binding components arise in the  contact region.
\end{Remark}

We can now give a proof of the second part of Theorem~\ref{thm:mainBirk}.

\begin{theorem}\label{thm:Bsection}
Let $(\lambda,\omega)$ be a SHS in $\mathcal{B}_s$ on a closed 3-manifold $M$ whose Reeb vector field is $X$. 
 Then there is a $C^2$-neighborhood of $(\lambda,\omega)$ in the set of SHS, such that every Reeb vector field of a SHS in this neighborhood admits a $\partial$-strong Birkhoff section. 
\end{theorem}

\begin{proof}
Let us keep the notation of the proof of Theorem~\ref{thm:bbook}. If the Reeb vector field is strongly non-degenerate for some structural decomposition $U, N_0, N_c$, it is also strongly non-degenerate in the closed 3-manifold $\overline W$ obtained by blowing down the boundary components of $N_c$. By \cite[Theorem A]{CM},  the $\partial$-strong broken book decomposition adapted to the Reeb vector field obtained in \cite[Theorem 1.1]{CDR} can be modified into a proper rational open book decomposition adapted to the Reeb vector field, whose pages are $\partial$-strong transverse surfaces. Applying Theorem~\ref{thm:bbook} and Remark~\ref{rem:bbook}, we deduce that the Reeb vector field of $(\lambda,\omega)$ is supported by a rational open book decomposition. One of its pages defines a $\partial$-strong Birkhoff section of the Reeb vector field.

Finally, we want to show that given any SHS $(\lambda,\omega)$ satisfying our hypotheses, there is $C^2$-open neighborhood of stable Hamiltonian structures around it that also admit a $\partial$-strong Birkhoff section. We argue exactly as in the last step of the proof of Theorem~\ref{thm:bbook}. The $\partial$-strong Birkhoff section that we constructed is such that in the contact regions  the binding components are non-degenerate periodic orbits. Then \cite[Proposition 5.1]{CDHR} shows that any Reeb vector field sufficiently $C^1$-close to the Reeb vector field of a strongly non-degenerate contact form also admits a $\partial$-strong Birkhoff section. Hence the Reeb vector field of
$(\widetilde \lambda, \widetilde \omega)$ admits a $\partial$-strong Birkhoff section in the contact region. The slope in the integrable region is non-constant in each connected component, so by Remark~\ref{rem:bbook} and Theorem~\ref{thm:bbook} we can still construct a $\partial$-strong Birkhoff section for the Reeb vector field of $(\widetilde \lambda, \widetilde \omega)$.
\end{proof}

The combination of Theorems~\ref{thm:bbook} and \ref{thm:Bsection} yields the first two parts of Theorem~\ref{thm:mainBirk}. The last one, namely the fact that $\mathcal{B}_s$ is $C^1$-dense among the set of SHS, is proved in \textsection\ref{sec:gen}.

\section{Genericity of Birkhoff sections}\label{sec:gen}

In this section we use our previous results to show that there exists a $C^1$-open and $(C^1,C^\infty)$-dense set of SHS that admits a Birkhoff section. Here the $(C^1,C^\infty)$-topology means that in the set of SHS, we measure the distance between the stabilizing 1-forms in the $C^1$-topology and the distance between the 2-forms in the $C^\infty$-topology. As a corollary, we deduce Theorem~\ref{thm:generic}.

\subsection{Density of $\mathcal{B}_s$
}\label{ss:density}

We first show that strongly contact non-degenerate SHS whose slope is non-constant in each connected component of the integrable region $U$ are dense, in a suitable topology, in the set of SHS. The following statement is the missing part of Theorem~\ref{thm:mainBirk}.

\begin{theorem}\label{thm:density}
Let $(\lambda,\omega)$ be a SHS on a closed 3-manifold $M$. Then there exists a SHS $(\widetilde \lambda, \widetilde \omega) \in \mathcal{B}_s$ such that $\widetilde \lambda$ is arbitrarily $C^1$-close to $\lambda$ and $\widetilde \omega$ is arbitrarily $C^\infty$-close to $\omega$.
\end{theorem}

If there is a structural decomposition of $M$ with respect to $(\lambda, \omega)$ with $U$ empty, the result follows from known results for Reeb vector fields and suspensions. We give a proof in the case $U\neq \emptyset$. Under this assumption an intermediate step is the following lemma.

\begin{lemma}\label{lem:nonconstantslope}
Let $(\lambda,\omega)$ be a stable Hamiltonian structure on $M$ and $N,U$ a structural decomposition. There exists an arbitrarily $C^\infty$-small perturbation $(\widetilde \lambda,\widetilde \omega)$ of $(\lambda,\omega)$, compactly supported in the integrable region $U$ and $T^2$-invariant, that yields a stable Hamiltonian structure with structural decomposition $N,U$ such that the slope of $\ker \widetilde \omega$ is non-constant in each connected component of $U$.
\end{lemma}

\begin{proof}
In any connected component of $U$, we take coordinates $(x,y,t)$ in $T^2\times I$. We know that $\lambda$ is of the form
$$ \lambda= g_1(t) dx + g_2(t) dy + g_3(t)dt,$$
and $\omega$ is of the form
$$ \omega= h_1(t)dt\wedge dx + h_2(t)dt\wedge dy.$$
The fact $(\lambda,\omega)$ defines a stable Hamiltonian structure implies that
$${g_1}'h_2-{g_2}'h_1=0 \qquad \mbox{and} \qquad  h_1g_2-h_2g_1>0. $$
If $\ker \omega$ has constant slope, then $\frac{h_2}{h_1}$ is constant. We might simply perturb the function $h_2$ to $\widetilde h_2(t)$ in a way that $h_2(t)=\widetilde h_2(t)$ for $|t-\frac{1}{2}|\geq \varepsilon$, the slope is no longer constant and $\widetilde h_2$ is arbitrarily $C^\infty$-close to $h_2$. This defines an arbitrarily small $C^\infty$ perturbation of $\omega$, that we denote by
$$ \widetilde \omega= h_1(t) dt\wedge dx + \widetilde h_2(t) dt\wedge dy.$$
We can now perturb $\lambda$ by a $C^\infty$-small perturbation to
$$\widetilde \lambda= g_1(t)dx+ \widetilde g_2(t)dy+g_3(t)dt,$$
where $\widetilde g_2(t)$ is determined by the equations
\begin{equation*}
\begin{cases}
\widetilde {g_2}'=\frac{{g_1}'\widetilde h_2}{h_1},\\
\widetilde g_2(0)=g_2(0).
\end{cases}
\end{equation*} 
The function $\widetilde g_2$ coincides with $g_2$ for $t\leq \frac{1}{2}-\varepsilon$. It  coincides with $g_2$ for $t > \frac{1}{2}+\varepsilon$ if and only if 
$$ \int_0^1 \frac{{g_1}'\widetilde h_2}{h_1}dt=g_2(1)-g_2(0).  $$
It is clear that we can choose some $\widetilde h_2$ satisfying this condition (see also \cite[Lemma 3.15]{CV1} for more general statements on this type of perturbations, considered in the $C^1$-topology). It follows that $(\widetilde \lambda,\widetilde \omega)$ is a stable Hamiltonian structure that is $C^\infty$-close to $(\lambda,\omega)$, coincides with it outside of $V=T^2\times(1/2-\varepsilon,1/2+\varepsilon)\subset T^2\times I$ and whose Reeb vector field has a non-constant slope in $V$. The perturbation can hence be applied in each connected component of $U$.
\end{proof}

It follows from the previous lemma and the fact that contact non-degenerate SHS are $C^1$-dense \cite[Theorem 4.6]{CV1}, that $\mathcal{B}$ is $C^1$-dense among SHS.  The techniques in \cite[Theorem 4.6]{CV1} allow us to prove that $\mathcal{B}_s$ is dense with $C^1$-topology in the 1-form $\lambda$ and $C^\infty$-topology in $\omega$ as we explain below. 

\smallskip

\noindent {\it Proof of Theorem~\ref{thm:density}.}
Let $(\lambda,\omega)$ be a SHS on a closed 3-manifold $M$. By Theorem \ref{thm:struc}, there exists some $\lambda'$ arbitrarily $C^1$-close to $\lambda$ such that the SHS $(\lambda',\omega)$ admits a structural decomposition $U, N_0, N_c$ such that $d\lambda'=c\omega$ for some constant $c\in \mathbb{R}$ in each connected component of the contact region $N_c$, and $d\lambda'=0$ in $N_0$. We can perturb the SHS by a perturbation of the form
$$(\lambda'+ \eta, c(\omega+d\eta)),$$ 
where $\eta$ is a 1-form that is arbitrarily $C^\infty$-small and compactly supported in the interior of the contact region that makes the stable Hamiltonian structure contact strongly non-degenerate \cite{CM, Ro, Pe}. This is possible because close to the boundary of each connected component of $N$ the flow is just integrable of constant irrational slope. By Lemma~\ref{lem:nonconstantslope}, we can make another arbitrarily $C^\infty$-small perturbation, compactly supported in the integrable region, such that the slope of the Reeb vector field in each connected component of $U$ is non-constant. This shows that $\mathcal{B}_s$ is dense with respect to the $C^1$-topology in the 1-form and the $C^\infty$-topology in the 2-form. \hfill $\square$


\begin{Remark}\label{rm:rat}
In \cite{CV0}, it was shown that any SHS is exact stable homotopic (through an homotopy that is not small in any $C^k$-topology) to one which is supported by an open book decomposition, which is equivalent to the Reeb vector field admitting a Birkhoff section. Given a SHS $(\lambda, \omega)$, the $C^1$-perturbation in the proof above can be done through an exact stable homotopy (as per Lemma~\ref{lem:nonconstantslope} and \cite[Theorem 4.6]{CV1}). Hence, there is $C^1$-small exact stable homotopy $(\lambda+\alpha_t, \omega+ d\mu_t)$ such that $\alpha_0=0$, $\mu_0=0$ and such that the Reeb vector field of $(\lambda+\alpha_1, \omega+d\mu_1)$ admits a Birkhoff section. Equivalently, this means that $(\lambda+\alpha_1,\omega+d\mu_1)$ is supported by a rational open book decomposition. 
\end{Remark}

\subsection{Open and dense sets with Birkhoff sections}
Using the previous dense set, we proceed to show that a $C^1$-generic SHS admits a Birkhoff section and deduce Theorem~\ref{thm:generic}.

\begin{theorem}\label{thm:C1generic}
Let $M$ be a closed 3-manifold. In the set of SHS on $M$, there is a set of SHS whose Reeb vector fields admit a $\partial$-strong Birkhoff section that is dense in the $(C^1,C^\infty)$-topology, and $C^1$-open.
\end{theorem}
{We point out that the set of SHS that we prove to have such a Birkhoff section is not $\mathcal{B}_s$. We start with any SHS, and perturb it to be in $\mathcal{B}_s$ but then we need an extra perturbation of the SHS to ensure that the Birkhoff section is $\partial$-strong.}

\begin{proof}
Let $(\lambda',\omega')$ be a SHS. By Theorem~\ref{thm:density}, there exists an arbitrarily $(C^1,C^\infty)$-close SHS $(\lambda,\omega) \in \mathcal{B}_s$. By Theorem~\ref{thm:Bsection}, the Reeb vector field $X$ of $(\lambda,\omega)$ admits a Birkhoff section $S$. This Birkhoff section is $\partial$-strong and the binding orbits are all non-degenerate except those in the integrable region $U$, which is a union of regular level sets of $f=\frac{d\lambda}{\omega}$. If the Birkhoff section is  a global section (meaning that $\partial S=\emptyset$), then this is $C^1$-robust with respect to perturbations of the vector field and we are done. Similarly, if there are no binding orbits in the integrable region, we conclude by \cite[Section 5]{CDHR}. 

In general, let $\gamma_1,..., \gamma_r$ be the binding orbits of $S$ in $U$. For each $\gamma\in \{\gamma_1,..., \gamma_r\}$, recall that this orbit lies in a torus fiber $T$ of $U$ given by a connected component of a regular level set of $f$, and thus there is a neighborhood $\mathcal{V}$ of  $T$  such that $\mathcal{V}=T^2\times (-\varepsilon,\varepsilon)$ with $f=c_i+t$, where $c_i\neq 0$ is the (regular) value of $f$ on $T$, and $t$ is the coordinate in the second factor of $\mathcal{V}$. Let $Z=\sqcup_{i=1}^r c_i$. By \cite[Proposition 3.23]{CV1}, there exist an arbitrarily $C^1$-close 1-form $\widetilde \lambda$ such that $(\widetilde \lambda, \omega)$ is a SHS, and the function $\widetilde f=\frac{d\widetilde \lambda}{\omega}$ satisfies $\widetilde f\equiv c_i$ in an open neighborhood of $f^{-1}(c_i)$. Notice that the Reeb flow of $(\widetilde \lambda, \omega)$ is just a reparametrization of $X$ by a function that is arbitrarily $C^1$-close to the constant function equal to $1$. In particular, taking the open neighborhood $\mathcal{V}$ small enough, the SHS in $\mathcal{V}\subset \widetilde f^{-1}(c_i)$ is of the form $(\widetilde \lambda, c_i d\widetilde \lambda)$. 

Choose $\gamma \in \{\gamma_1,\ldots, \gamma_r\}$. We claim that there exists a function $\rho$ arbitrarily $C^\infty$-close to $1$ and equal to $1$ outside a compact subset of $\mathcal V$ such that the Reeb flow of $(\rho\widetilde \lambda, c_i d (\rho \widetilde \lambda))$ has $\gamma$ as a non-degenerate periodic orbit. Observe that the SHS $(\rho\widetilde \lambda, c_i d (\rho \widetilde \lambda))$ extends as $(\widetilde \lambda, \omega)$ away from $\mathcal V$.

To see that such a function $\rho$ exists, consider $\mathcal V$ equipped with the contact form $\widetilde\lambda$, and the symplectization $\mathcal V\times (-\delta,\delta)$ with symplectic form $\Omega= d(e^s\widetilde\lambda)$, where $s$ is the parameter in the second factor. We identify $\mathcal V$ with $\mathcal V\times \{0\}$. By \cite[Lemma 19]{Ro}, there exists a hypersurface $\widetilde{\mathcal V}$ that is arbitrarily $C^\infty$-close to $\mathcal V\times \{0\}$, coincides with $\mathcal V\times \{0\}$ away from a small neighborhood of $\gamma\subset \mathcal V$, and such that $\gamma$ is a non-degenerate periodic orbit of the Hamiltonian flow induced on $\widetilde V$. It is well known that the vector field induced on any hypersurface $C^\infty$-close to $\mathcal V\times \{0\}$ is the Reeb vector field of the contact form $\rho \widetilde\lambda$ for some function $\rho$ that is $C^\infty$-close to $1$ (see e.g. \cite[Section 6]{HWZ2}). Applying this fact to $\widetilde{\mathcal V}$, we find a contact form $\rho \widetilde \lambda$, where the function $\rho$ is equal to $1$ outside of a small neighborhood of $\gamma$, whose Reeb flow has $\gamma$ as a non-degenerate periodic orbit. Denote by $(\hat \lambda, \hat \omega)$ the SHS obtained by extending $(\rho \widetilde \lambda, c_i d(\rho \widetilde \lambda))$ by $(\widetilde \lambda,\widetilde \omega)$ in $M\setminus \mathcal V$, it is $C^\infty$-close to $(\widetilde \lambda,\widetilde \omega)$. Its Reeb flow $\hat R$ coincides along $V$ with that of the contact form $f\widetilde \lambda$, and hence has $\gamma$ as a non-degenerate periodic orbit. Doing this at each $\gamma_i$, we find a stable Hamiltonian structure that we still denote by $(\widetilde \lambda, \widetilde \omega)$ for which $S$ is a Birkhoff section with non-degenerate periodic orbits of the Reeb flow $\hat R$. Furthermore, since the Birkhoff section $S$ was $\partial$-strong, it remains a Birkhoff section of $\hat R$.

 Thus we have shown that there is SHS $(\hat \lambda,\hat \omega)$ that is arbitrarily $(C^1, C^\infty)$-close to $(\lambda',\omega')$ and whose Reeb vector field admits a $\partial$-strong Birkhoff section with non-degenerate binding orbits. Finally, admitting a $\partial$-strong Birkhoff section with non-degenerate binding orbits is a $C^1$-open condition among vector fields \cite[Proposition 5.1]{CDHR}, and $C^1$-close stable Hamiltonian structures define $C^1$-close Reeb vector fields. This concludes the proof.
\end{proof}


Although we did not explicitly state it, the $C^\infty$-perturbation of the two-form that we produce in Theorem \ref{thm:C1generic} is exact. We finish this section by proving Theorem~\ref{thm:generic}.

\begin{proof}[Proof of Theorem~\ref{thm:generic}]
Let $X$ be a reparametrized Reeb vector field of a SHS, i.e. it satisfies $\lambda(X)\neq 0$ and $\iota_X\omega=0$ for some SHS $(\lambda,\omega)$. Since admitting a Birkhoff section does not depend on the orientation of $X$, if $\lambda(X)<0$ we consider $-X$ instead. 

Theorem~\ref{thm:C1generic} shows that there exists a $C^1$-perturbation of $\lambda$ and a $C^\infty$-perturbation of $\omega$ such that the Reeb vector field of $(\widetilde \lambda, \widetilde \omega)$ admits a $\partial$-strong Birkhoff section with non-degenerate binding orbits. Let $\widetilde R$ be the Reeb vector field defined by such SHS, we claim that there exists a positive function $g\in C^\infty(M)$ such that $\widetilde X=g\widetilde R$ is a $C^\infty$-perturbation of $X$. Indeed, observe that $\omega$ and $\widetilde \omega$ are $C^\infty$-close, hence given a section of the bundle $\ker \omega$ (such as $X$), there exists a section $\widetilde X$ of $\ker \widetilde \omega$ that is $C^\infty$-close to $X$. Since $\iota_{\widetilde R}\widetilde \omega=0$, the section $\widetilde X$ is of the form $\widetilde X=g\widetilde R$ for some positive function $g$. We have thus shown that given any $X\in \mathcal{CSHS}(M)$, there exists a $C^\infty$-perturbation $\widetilde X$ of $X$ that admits a $\partial$-strong Birkhoff section with non-degenerate binding orbits. On the other hand, by \cite[Section 5]{CDHR}, this property holds for a $C^1$-neighborhood of $\widetilde X$. 
\end{proof}

\begin{Remark}\label{rem:fixvol}
    Notice that the $C^\infty$-perturbation $\widetilde X$ of $X$ in the previous proof is obtained by producing a $C^\infty$-perturbation $\widetilde \omega$ of an arbitrary given stabilizable 2-form $\omega$ such that $\iota_X\omega=0$. In particular, if $X$ preserves a given volume form $\mu$, and $\iota_X\mu$ is chosen as the stabilizable 2-form $\omega$, then we can define the vector field $\widetilde X$ by $\iota_{\widetilde X}\mu=\widetilde \omega$. In particular, if the set of vector fields that preserve $\mu$ is denoted by $\mathfrak{X}_\mu(M)$, the previous proof shows that the set of vector fields in $\mathcal{CSHS}(M)\cap \mathfrak{X}_\mu(M)$ that admit a Birkhoff section is dense in $\mathcal{CSHS}(M)\cap \mathfrak{X}_\mu(M)$.
\end{Remark}

\section{Analytic integrable Reeb flows}\label{sec:analytic}

In this last section, we analyze the case of SHS whose Reeb flow admits a real analytic first integral $f\in C^\omega (M)$. Every time that we speak of an analytic object, we mean real analytic with respect to a real analytic structure on $M$. Real analytic SHS have beed studied previously. Given an analytic $(\lambda,\omega)$, the function $\frac{d\lambda}{\omega}$  is an analytic first integral of the Reeb vector field of $(\lambda,\omega)$. As shown by 
Cieliebak and Volkov \cite{CV2}, a stationary solution to the Euler equations on a Riemannian 3-manifold $(M,g)$ with an analytic Bernoulli function gives rise to a SHS whose Reeb flow admits an analytic first integral.

\begin{defi}\label{def:Bsplit}
Let $X$ be a nonvanishing vector field on a closed 3-manifold $M$. We say that $X$ admits a \emph{Birkhoff splitting} if there exists a finite collection of invariant tori $T_1,...,T_k$ such that:
\begin{itemize}
\item[-] the manifold $\widetilde M$ obtained by cutting $M$ open along $T_1,...,T_k$  is connected,
\item[-] the induced vector field $\widetilde X$ on $\widetilde M$ admits a Birkhoff section.
\end{itemize}
\end{defi}
Observe that the existence of a Birkhoff splitting reduces the dynamics of $X$ to the dynamics of a homeomorphism of a connected compact surface with boundary, as a Birkhoff section does. One could drop the assumption that the manifold $\widetilde M$ is connected, in which case the existence of a Birkhoff splitting for analytic integrable SHS follows immediately from \cite{CV2} and Proposition~\ref{prop:BirkT2}. 

\begin{theorem}\label{thm:split}
Let $(\lambda,\omega)$ be a stable Hamiltonian structure whose Reeb vector field admits a non-constant analytic first integral. Then the Reeb vector field defined by $(\lambda,\omega)$ admits a Birkhoff splitting.
\end{theorem}

The proof of this theorem relies crucially in Sections 2 and 3 of \cite{CV2}. To avoid repeating several arguments, we refer to concrete parts of those sections. Throughout the proof, we will sometimes construct Birkhoff sections in domains of $M$ using the following trick due to Tischler \cite{Tis}: on a compact manifold with a non-vanishing vector field $X$, a closed 1-form $\alpha$ with rational cohomology class such that $\alpha(X)>0$ can be used to construct a global section of $X$. This is done by multiplying $\alpha$ by a large enough integer so that this multiple defines a fibration over $S^1$. Any fiber of this fibration will be a global section of the flow.

\begin{proof}
Let $f$ be the analytic first integral of the Reeb vector field $X$ of $(\lambda, \omega)$, which is not necessarily the function $\frac{d\lambda}{\omega}$. For each critical value $c_i$, let $S_{c_i}$ be the union of the connected components of $f^{-1}(c_i)$ that are not regular. Then let $N_i$ be the intersection of a small open neighborhood of $S_{c_i}$ with $f^{-1}([c_i-\delta,c_i+\delta])$, where $\delta$ is small enough so that $N_i$ is a compact submanifold with boundary, and the boundary is a union of invariant tori given by connected components of $f^{-1}(c_i \pm \delta)$. Let $U=\overline{M\setminus N}$ be the closure of the complement of $N=\bigsqcup_{c_i}N_i$. In $U$, the function $f$ is everywhere regular and so the Reeb vector field of $(\lambda,\omega)$ is orbit equivalent to a linear flow on each regular torus of $f$, e.g. by \cite[Theorem 3.3]{CV1}. Since $f$ is real analytic, it has a finite number of critical values and  $U$ decomposes as a finite union of domains each diffeomorphic to $T^2\times I$. 

We now associate to the decomposition $M=N\cup U$ a graph: each connected component of $N$ is a vertex and each connected component of $U$ is an edge. Since each connected component of $U$ is of the form $T^2\times I$, it has two boundary components that coincide with boundary components of $N$. Thus we put an edge between two vertices if there is a connected component of $U$ intersecting their boundaries. This graph is connected, can have cycles and is finite.

The strategy of the proof is as follows: we choose a vertex $v_1$ and find a global section of the flow in the corresponding manifold $N_1$. Then we  extend the section as a Birkhoff section along the connected components of $U$ that correspond to the edges having at least one endpoint in $v_1$. Next, for each of those edges, we try to extend the Birkhoff section on the component $N_j$ of $N$ that corresponds to the other endpoint. Here we need to consider several cases (for example, both endpoints of the edge could be $v_1$) and it is not always possible to extend the Birkhoff section. We either extend the Birkhoff section to $N_j$ or cut along the boundary torus of $\partial N_j$ that corresponds to that edge. If we cut along a torus, we obtain a new manifold $\widetilde M_1$ with an associated graph that is obtained by removing the corresponding edge from the original graph. Recursively, we continue the construction along the vertices that are at distance two from $v_1$. As we will see in the proof, the only cases where we might have to cut along a torus can occur when we reach a vertex corresponding to a component of $N$ where we already constructed a Birkhoff section, i.e. when we circulate along a cycle of the graph. Hence, possibly after cutting open some cycles of the graph, the graph resulting after extending the Birkhoff section to all the vertices is always connected. We end up with a Birkhoff splitting.  

\smallskip

Following \cite{CV2}, let us recall the structure of $S_{c_i}$, that is by hypothesis a real analytic set and thus it is stratified with respect to dimension. By \cite[Proposition 2.6]{CV2}, each $S_{c_i}$ has finitely many connected components and each one is of one of the following forms:
\begin{enumerate}
\item a two-torus,
\item a Klein bottle, 
\item a periodic orbit,
\item a union of periodic orbits (the 1-dimensional strata) and open cylinders and M\"obius strips that are embedded (the 2-dimensional strata). The closure of the open cylinders and M\"obius strips are $C^1$-immersed surfaces whose boundary components are finite coverings of the 1-dimensional strata in $S_{c_i}$.
\end{enumerate}
 Given a 1-dimensional substratum $\gamma$ of $S_{c_i}$ (that is a periodic orbit of $X$), we denote by $d_{\gamma}\in \mathbb{N}$ the folding number of the covering of $\gamma$ by the 2-dimensional cylinders and M\"obius bands, i.e. the number of connected components of $\mathcal{U}\setminus \gamma$ where $\mathcal{U}$ is a small enough neighborhood of $\gamma$ inside $S_{c_i}$.
 
 The complement of $S_{c_i}$ is a finite union of regular integrable domains diffeomorphic to $T^2\times (0,1)$. Let $V$ be connected component of this set, whose boundaries are contained in $f^{-1}(c_1)\cup f^{-1}(c_2)$ with $c_1,c_2$ critical values, not necessarily distinct. Then by \cite[Lemma 3.4]{CV2}, if the slope of $X$ is constant in $V$ then:
\begin{itemize}
\item[-] If the slope is rational, each $\overline{V}\cap f^{-1}(c_1)$ and $\overline{V}\cap f^{-1}(c_2)$ is either a periodic orbit, a rational torus, a Klein bottle or a union of periodic cylinders and M\"obius bands connected along their boundaries (possibly multiply covered).
\item[-] If the slope is irrational, each $\overline{V}\cap f^{-1}(c_1)$ and $\overline{V}\cap f^{-1}(c_2)$ is either an isolated periodic orbit or an irrational torus.
\end{itemize}
Observe that in the rational case, the set $\overline{V}\cap f^{-1}(c_1)$ (or $\overline{V}\cap f^{-1}(c_2)$) is not necessarily a whole connected component of the level set $f^{-1}(c_i)$, but instead a connected substratum of it. 

Our main tool follows from the proof of \cite[Proposition 2.8]{CV2}.
 
\begin{prop}\label{prop:CV2}
Let $\eta_0$ be a closed 1-form defined in an open neighborhood of a closed stratified subset $K_i$ of $S_{c_i}$ such that $\eta_0(X)>0$ and $\int_{\gamma}\eta_0=\frac{1}{d_{\gamma}}$ for each 1-dimensional substratum $\gamma$ (which is a periodic orbit) of $K_i$. Then $\eta_0$ extends to a closed-one form $\eta$ in all $N_i$ satisfying $\eta(X)>0$ and $\int_{\gamma} \eta=\frac{1}{d_{\gamma}}$ for each 1-dimensional substratum $\gamma$ of $S_{c_i}$.
\end{prop}

If $\eta_0$ is not prescribed near $K_i$ or $K_i=\emptyset$, Proposition~\ref{prop:CV2}  is exactly the statement of \cite[Proposition 2.8]{CV2} that provides an $\eta$ in $N_i$ satisfying the claimed properties.

\begin{proof}
Here a closed stratified subset $K_i$ will always be a collection of open cylinders, open M\"obius bands, and the circles in their closure. 

 From the proof of \cite[Proposition 2.8]{CV2}, we deduce  that one can prescribe $\eta_0$ near a connected substratum. Indeed, given $\eta_0$, we choose some closed one form $\eta_2$ near each 1-dimensional substratum $\gamma$ that is not in $K_i$, such that $\eta_2(X)>0$ and $\int_{\gamma}\eta_2=\frac{1}{d_{\gamma}}$. Now we have a 1-form defined near each 1-dimensional substratum of $S_{c_i}$ and in the whole closed stratified subset $K_i\subset S_{c_i}$. It is shown in \cite[Section 2, p. 463-464]{CV2} how to extend a given 1-form satisfying these two conditions from a neighborhood of the boundary of a two-stratum to a neighborhood of its interior. Hence, we apply this to each two-stratum that was not in $K_i$, to conclude Proposition~\ref{prop:CV2}.
\end{proof}

We start with the vertex $v_1$ that corresponds to a connected component $N_1$ of $N$. Proposition~\ref{prop:CV2}, with $K_1=\emptyset$, gives a closed 1-form $\eta_1$ such that $\eta_1(X)>0$ in $N_1$. Up to perturbing it, we can assume that $\eta_1$ defines a rational cohomology class so that it can be used to construct a global section $\Sigma_1$ (a surface with boundary) of $X$ in $N_1$. Observe that $\partial\Sigma_1=\Sigma_1\cap \partial N_1$.

We now consider the graph associated to the structural decomposition $N$, $U$ and choose an edge $e_1$  incident on $v_1$. Then $e_1$ corresponds to a connected component $U_1\cong T^2\times I\subset U$. Observe that $U_1\subset V_1$ with $V_1$ a connected component  of $M\setminus \left( \bigcup S_{c_i} \right)$. Observe that $e_1$ might have its two endpoints in $v_1$. Since $\Sigma_1$ is defined in $N_1$, it induces near at least one  boundary component of $U_1$ a section of the flow: that is, a collection of circles in each torus fiber of $U_1$ near that boundary component.  Denote by $N_2$ the connected component of  $N$ that corresponds to the other endpoint of $e_1$. It could happen that $N_2=N_1$. 

Three things can occur: the slope of $X$ in $V_1$ is non-constant, constant rational or constant irrational. If it is non-constant in $V_1$, we can assume it is non-constant in $U_1$ up to shrinking a bit the domains $N_1$ or $N_2$ and enlarging $U_1$. 

\smallskip

\textbf{The slope is non-constant.} We can use Proposition~\ref{prop:BirkT2} to find a Birkhoff section in $U_1$ that restricts to each torus fiber near the boundary component of $U_1\cap N_2$ as a given family of curves. Notice that the binding orbits will be $\partial$-strong. Hence, if $N_2=N_1$, then $\Sigma_1$ is defined near each boundary component of $U_1$ and we can find a Birkhoff section that extends $\Sigma_1$ inside $U_1$. If $N_2\neq N_1$, we apply Proposition~\ref{prop:CV2} to $N_2$, finding a closed form $\eta_2$ near $N_2$ that evaluates positively on the Reeb vector field and can be assumed to define a rational cohomology class in $N_2$. In particular, a fiber of the fibration defines global section $\Sigma_2$, which is defined near $U_1\cap N_2$. We can glue $\Sigma_1$ and $\Sigma_2$ with a Birkhoff section inside $U_2$ using Proposition~\ref{prop:BirkT2}, adding $\partial$-strong binding components.

\smallskip

\textbf{The slope is constant irrational.} In this case, both $N_1$ and $N_2$ are diffeomorphic to a neighborhood of irrational tori or to a neighborhood of an isolated periodic orbit. Thus $N_1$ and $N_2$ are either diffeomorphic to $T^2\times I$ or to $D^2\times S^1$. 

If $N_2=N_1$, then $N_1$ has at least two boundary components and thus $N_1\cong T^2\times I$. Then  $U=U_1$ diffeomorphic to $T^2\times I$ and $M$ is  obtained by gluing $N_1$ and $U$ along their boundaries. Cutting along one of the boundary components of $N_1$, we obtain $\widetilde{M}\cong T^2\times I$ and the flow is irrational in each fiber. Any homology class defines a global section, which induces a Birkhoff splitting for $X$.

Assume that $N_2\neq N_1$, we can begin by extending $\Sigma_1$ over $U_1$. If $N_2$ is a neighborhood of an irrational torus, then $\Sigma_1$ trivially extends to $N_2$. 
If $N_2\cong D^2\times S^1$, the section $\Sigma_1$ induces a collection of embedded circles in $\partial N_2$. If these circles are not meridian circles, then $\Sigma_1$ extends to a $\partial$-strong Birkhoff section in $D^2\times S^1$ whose binding is the core orbit $\{0\}\times S^1$. Otherwise $\Sigma_1\cap D^2\times S^1$ induces circles which are up to isotopy of the form $\{r=ct\}\times S^1 \times \{q\}$. In this last case, we simply extend $\Sigma_1$ as the disk $D^2\times \{q\}$, which is transverse to the flow, and hence no new binding components arise. Observe that we can isotope the section so that it is smooth.  

\smallskip

\textbf{The slope is constant rational.} As in the previous $\Sigma_1$ extends trivially to $U_1$. If $N_2=N_1$, we cut along one of the connected components of $\partial N_1$, as in the previous case.

If $N_2\neq N_1$, then we proceed as follows. Let $\eta_1$ be the rational closed 1-form defined in $N_1\cup U_1$ dual to $\Sigma_1$, it is then defined near one boundary component of $N_2$ and $\eta_1(X)>0$. Then $\eta_1$ can be further  extended to  $\overline V_1$. Consider now a neighborhood of $\overline{V}_1$  which contains a connected substratum $K_2$ of $S_{c_2}\cap N_2$ that is not necessarily equal to the connected set $S_{c_2}\cap N_2$.

We want to apply Proposition~\ref{prop:CV2} to extend $\eta_1$ to all $N_2$. Choose a 1-dimensional substratum $\gamma$ of $K_2$, we can choose a positive constant $c$ such that $c\eta_1$ satisfies $\int_\gamma c\eta_1=\frac{1}{d_{\gamma}}$. Let $\kappa$ be another 1-dimensional substratum of $K_2$, we claim that $\int_{\kappa} c\eta_1=\frac{1}{d_{\kappa}}$. Indeed, as argued in the beginning proof of \cite[Proposition 3.2]{CV2}, the homology class of $d_{\gamma}\gamma$ and $d_{\kappa}\kappa$ are both equal to that of any simple closed orbit of the integrable region $U_1$. Hence $[d_{\gamma}\gamma]=[d_{\kappa}\kappa]$, and since $\eta_1$ is closed it follows that $\int_{\kappa} c\eta_1= \frac{d_{\gamma}}{d_{\kappa}} \int_\gamma c\eta_1=\frac{1}{d_{\kappa}}$. We can apply Proposition~\ref{prop:CV2} to $c\eta_1$, obtaining as a byproduct an extension as well of $\eta_1$ to a closed 1-form positive on $X$ near $N_1\cup U_1\cup N_2$. If $\eta_1$ does not define a rational cohomology class in $N_1\cup U_1\cup N_2$, we perturb it so that it does and still evaluates positively on the Reeb vector field. A fiber of the fibration defined by this new 1-form is a global section of the Reeb vector field in $N_1\cup U_1\cup N_2$. We keep the notation $\eta_1$ for this rational 1-form on $N_1\cup U_1\cup N_2$.

\smallskip

We have now a $\partial$-strong Birkhoff section in the domain $N_1\cup U_1\cup N_2$ that is either connected or disconnected according to the cases above. We want to iterate the construction following a different edge of $v_2$, corresponding to a domain $U_2$ attached to $N_1\cup U_1\cup N_2$.

Let $N_3$ be the connected component of $N$  corresponding to the other vertex of $v_2$. In $N_1\cup U_1\cup N_2$ we either have a section dual to the 1-form $\eta_1$ or a $\partial$-strong Birkhoff section. In the second case, we start by blowing up the binding orbits to obtain a section in the resulting manifold with boundary. In both cases, we denote the corresponding manifold with boundary $\widetilde{M}_1$ endowed with a rational 1-form $\eta_1$.
All the arguments we used for $(N_1,\eta_1)$ apply to $(\widetilde{M}_1,\eta_1)$: if $N_3$ intersects $N_1\cup U_1\cup N_2$ and in $U_2$ the slope of $X$ is constant we cut along an invariant torus  and if it not we extend $\eta_1$. 

Doing this for each edge of the graph we end up 
with a Birkhoff section of $X$ in a manifold $\widetilde M$ that is connected and obtained by cutting $M$ along a finite number of invariant tori. Thus, we have shown that $X$ admits a Birkhoff splitting.
\end{proof}

{
\begin{Corollary}
    Let $(H,f)$ be an integrable system on a 4-dimensional symplectic manifold $(W,\Omega)$. Let $M$ be a (connected component) of a regular energy level set of $H$ such that $f|_M$ is an analytic non-constant function in $M$. If the Hamiltonian flow of $H$ restricted to $M$ has no invariant Reeb cylinders, then it admits a Birkhoff splitting in $M$.
\end{Corollary}
\begin{proof}
   The analysis carried out in \cite[Sections 2 and 3]{CV2} shows that $M$ is a stable energy level set: the Hamiltonian flow is parallel to the Reeb flow of a SHS on $M$. It has an analytic integral, so by Theorem~\ref{thm:split} it admits a Birkhoff splitting.
\end{proof}
}

\end{document}